\documentclass[reqno,12pt]{amsart}
\usepackage[active]{srcltx}
\usepackage {color}
\usepackage{amsmath,amsthm,amssymb}
\usepackage{epsfig}
\usepackage{graphicx}
\usepackage{amssymb}
\usepackage{latexsym}
\usepackage{pdfpages}

\usepackage{enumitem}

\usepackage{ulem} 

\usepackage{ifpdf}
\newcommand{\res}{\!\!\mathop{\hbox{
			\vrule height 7pt width .5pt depth 0pt
			\vrule height .5pt width 6pt depth 0pt}}
	\nolimits}

\def\z{{\bf z}}

\ifpdf 
\usepackage[hidelinks]{hyperref}
\else 
\fi 

\usepackage[usenames,dvipsnames]{pstricks}
\usepackage{epsfig}
\usepackage{pst-grad} 
\usepackage{pst-plot} 

\usepackage{color}

\newtheorem{theorem}{Theorem}[section]

\newtheorem{definition}[theorem]{Definition}
\newtheorem{proposition}[theorem]{Proposition}
\newtheorem{corollary}[theorem]{Corollary}
\newtheorem{remark}[theorem]{Remark}
\newtheorem{example}[theorem]{Example}
\newtheorem*{theorem*}{\it Theorem}

\def\vint_#1{\mathchoice%
	{\mathop{\kern 0.2em\vrule width 0.6em height 0.69678ex depth -0.58065ex
			\kern -0.8em \intop}\nolimits_{\kern -0.4em#1}}%
	{\mathop{\kern 0.1em\vrule width 0.5em height 0.69678ex depth -0.60387ex
			\kern -0.6em \intop}\nolimits_{#1}}%
	{\mathop{\kern 0.1em\vrule width 0.5em height 0.69678ex
			depth -0.60387ex
			\kern -0.6em \intop}\nolimits_{#1}}%
	{\mathop{\kern 0.1em\vrule width 0.5em height 0.69678ex depth -0.60387ex
			\kern -0.6em \intop}\nolimits_{#1}}}
\def\vintslides_#1{\mathchoice%
	{\mathop{\kern 0.1em\vrule width 0.5em height 0.697ex depth -0.581ex
			\kern -0.6em \intop}\nolimits_{\kern -0.4em#1}}%
	{\mathop{\kern 0.1em\vrule width 0.3em height 0.697ex depth -0.604ex
			\kern -0.4em \intop}\nolimits_{#1}}%
	{\mathop{\kern 0.1em\vrule width 0.3em height 0.697ex depth -0.604ex
			\kern -0.4em \intop}\nolimits_{#1}}%
	{\mathop{\kern 0.1em\vrule width 0.3em height 0.697ex depth -0.604ex
			\kern -0.4em \intop}\nolimits_{#1}}}

\def\R{\mathbb R}
\def\N{\mathbb N}

\numberwithin{equation}{section}

\def\RR{{\mathbb{R}}}

\def\NN{{\mathbb{N}}}
\def\D{\displaystyle}

\def\1{\raisebox{2pt}{\rm{$\chi$}}}

\def\Xint#1{\mathchoice
	{\XXint\displaystyle\textstyle{#1}}%
	{\XXint\textstyle\scriptstyle{#1}}%
	{\XXint\scriptstyle\scriptscriptstyle{#1}}%
	{\XXint\scriptscriptstyle\scriptscriptstyle{#1}}%
	\!\int}
\def\XXint#1#2#3{{\setbox0=\hbox{$#1{#2#3}{\int}$}
		\vcenter{\hbox{$#2#3$}}\kern-.5\wd0}}

\def\dashint{\Xint-}

\def\Xint#1{\mathchoice
	{\XXint\displaystyle\textstyle{#1}}%
	{\XXint\textstyle\scriptstyle{#1}}%
	{\XXint\scriptstyle\scriptscriptstyle{#1}}%
	{\XXint\scriptscriptstyle\scriptscriptstyle{#1}}%
	\!\int}
\def\XXint#1#2#3{{\setbox0=\hbox{$#1{#2#3}{\int}$}
		\vcenter{\hbox{$#2#3$}}\kern-.5\wd0}}

\def\dashint{\Xint-}

\newcommand{\threepartdef}[6]
{
	\left\{
	\begin{array}{lll}
		#1 & \mbox{if } #2 \\
		#3 & \mbox{if } #4 \\
		#5 & \mbox{if } #6
	\end{array}
	\right.
}

\definecolor{violet(ryb)}{rgb}{0.53, 0.0, 0.69}
\usepackage{mathtools}
\mathtoolsset{showonlyrefs}

\begin{document}
	
\title[Least Gradient Functions]{\bf Least Gradient Functions in  Metric Random Walk Spaces}
	
\author[W. G\'{o}rny and J. M. Maz\'on]{Wojciech G\'{o}rny and Jos\'e M. Maz\'on}
	
\address{ W. G\'{o}rny: Faculty of Mathematics, Informatics and Mechanics, University of Warsaw, Banacha 2, 02-097 Warsaw, Poland,
		\hfill\break\indent
		{\tt  w.gorny@mimuw.edu.pl }}
	
\address{J. M. Maz\'{o}n: Departamento de An\`{a}lisis Matem\`atico,
		Universitat de Val\`encia, Dr. Moliner 50, 46100 Burjassot, Spain.
		\hfill\break\indent
		{\tt mazon@uv.es }}
	
%
%

\keywords{Random walk, Least gradient functions, Total variation flow, Functions of bounded variation.\\
\indent 2010 {\it Mathematics Subject Classification:} 05C81, 35R02, 26A45, 05C21, 45C99.}
	
\setcounter{tocdepth}{1}
	
%

\begin{abstract}
In this paper we study least gradient functions in metric random walk spaces, which include as particular cases the least gradient functions on locally finite weighted connected graphs and nonlocal least gradient functions on $\mathbb{R}^N$. Assuming that a Poincar\'{e} inequality is satisfied, we study the Euler-Lagrange equation associated with the least gradient problem. We also prove the Poincar\'e inequality in a few settings.
\end{abstract}
	
\maketitle
%
%
%
%
{\renewcommand\contentsname{Contents }
	\setcounter{tocdepth}{3}
	\tableofcontents
}

\section{Introduction and Preliminaries}
	
A  metric random walk space $[X,d,m]$ is a Polish metric space  $(X,d)$ together with a family  $m = (m_x)_{x \in X}$ of probability measures that encode the jumps of a Markov chain. Important examples of metric random walk spaces are: locally finite weighted connected graphs, finite Markov chains and $[\R^N, d, m^J]$ with $d$ the Euclidean distance and
$$m^J_x(A) :=  \int_A J(x - y) d\mathcal{L}^N(y) \quad \hbox{ for every Borel set } A \subset  \R^N \ ,$$
where $J:\R^N\to[0,+\infty[$ is a measurable, nonnegative and radially symmetric function with $\int_{\mathbb{R}^N} J \, dx = 1$.

\medskip

In the Euclidean space, least gradient problems are closely related to the study of minimal surfaces. They first appeared in the pioneering work of Bombieri et al. \cite{BdGG}, where the authors show that the boundaries of superlevel sets of a function of least gradient are area-minimizing in the sense that the characteristic functions of those sets are also functions of least gradient (see also \cite{Mir} for stability results on functions of least gradient). Conversely, Sternberg, Williams, and Ziemer proved in \cite{SWZ} (see also \cite{SZ1}) existence of a function of least gradient with a given continuous trace by explicitly constructing each of its superlevel sets in such a way that they are area-minimizing and reflect the boundary condition. Due to their important applications in conductivity imaging, such problems (including anisotropic cases) have received extensive attention in the past decade (see for instance: \cite{Gor}, \cite{Gor2}, \cite{JMN}, \cite{MazRoSe}, \cite{Moradifam}, \cite{MNT}, \cite{NTT}).  In particular, in \cite{MazRoSe} the authors give a formulation of the problem of Euler-Lagrange type, where the structure of the solution is governed by a single divergence-measure vector field. We will utilise a similar approach in the nonlocal setting.

To our knowledge, the only results on least gradient problems for nonlocal operators are to one obtained in \cite{MazPeRoTo}. The aim of this paper is to study least gradient functions in the general setting of metric random walk spaces. As a particular case, we obtain results in the nonlocal least gradient problem on graphs. Moreover, some results obtained in this paper are new also in the context of  nonlocality defined by a continous nonnegative kernel on a Euclidean space.

The paper is organized as follows: in the first part of the paper, we provide the Euler-Lagrange formulation associated with the least gradient problem in general metric random walk spaces. Under the assumption that a nonlocal Poincar\'{e} type inequality is satified, we prove equivalence of the Euler-Lagrange equations with the variational formulation and study some properties of solutions to the least gradient problem in this setting. In the second part of the paper, we prove that such Poincar\'{e} inequality holds in many important examples of metric random walk spaces, such as the $\epsilon$-step random walk in weighted Euclidean spaces or Carnot-Carath\'eodory spaces and on locally finite graphs.  Moreover, we show that if the space does not satisfy the Poincar\'e inequality, then solutions in the Euler-Lagrange sense might not exist.

\subsection{Metric Random Walk Spaces}
	
Let $(X,d)$ be a  Polish metric  space equipped with its Borel $\sigma$-algebra.
A {\it random walk} $m$ on $X$ is a family of probability measures $m_x$ on $X$, $x \in X$, satisfying the two technical conditions: (i) the measures $m_x$  depend measurably on the point  $x \in X$, i.e., for any borelian  $A$ of $X$ and any borelian  $B$ of $\R$, the set $\{ x \in X \ : \ m_x(A) \in B \}$ is borelian; (ii) each measure $m_x$ has finite first moment, i.e. for some (hence any) $z \in X$, and for any $x \in X$ one has $\int_X d(z,y) dm_x(y) < +\infty$ (see~\cite{O}).
	
A {\it metric random walk  space} $[X,d,m]$ is a Polish metric space $(X,d)$ with a  random walk $m$.
	
Let $[X,d,m]$ be a metric random walk  space. A Radon measure $\nu$ on $X$ is {\it invariant} for the random walk $m=(m_x)$ if
$$d\nu(x)=\int_{y\in X}d\nu(y)dm_y(x),$$
that is, for any $\nu$-measurable set $A$, it holds that $A$ is $m_x$-measurable for $\nu$-almost all $x\in X$,  $\displaystyle x\mapsto  m_x(A)$ is $\nu$-measurable, and
$$\nu(A)=\int_X m_x(A)d\nu(x).$$
Hence, for any $u \in L^1(X, \nu)$, it holds that $u \in L^1(X, m_x)$ for $\nu$-a.e. $x \in X$, $\displaystyle x\mapsto \int_X u(y) d{m_x}(y)$ is $\nu$-measurable, and
$$\int_X u(x) d\nu(x) = \int_X \left(\int_X u(y) d{m_x}(y) \right)d\nu(x).$$
The measure $\nu$ is said to be {\it reversible} if, moreover, a more detailed balance condition
$$dm_x(y)d\nu(x)  = dm_y(x)d\nu(y)$$
holds true. Under suitable assumptions on the metric random walk space $[X,d,m]$, such an invariant and reversible measure $\nu$ exists and is unique, as we will see below. Note that the reversibility condition implies the invariance condition.
	
We will assume that the metric measure space $(X,d,\nu)$ is $\sigma$-finite.
	
\begin{example}\label{JJ}{\rm
(1) Consider $(\R^N, d, \mathcal{L}^N)$, with $d$ the Euclidean distance and $\mathcal{L}^N$ the Lebesgue measure. Let  $J:\R^N\to[0,+\infty[$ be a measurable, nonnegative and radially symmetric function verifying $\int_{\R^N}J(z)dz=1$. In $(\R^N, d, \mathcal{L}^N)$ we have the following random walk, starting at $x$,
$$m^J_x(A) :=  \int_A J(x - y) d\mathcal{L}^N(y) \quad \forall A \subset  \R^N \ \hbox{borelian}.$$
Applying Fubini's Theorem it is easy to see that the Lebesgue measure $\mathcal{L}^N$ is an invariant and reversible measure for this random walk. Hence, $[X, d, m^J]$ is a random walk space.
			
\noindent (2)  Let $K: X \times X \rightarrow \R$ be a Markov kernel on a countable space $X$, i.e.
$$K(x,y) \geq 0, \quad \forall x,y \in X, \quad \quad \sum_{y\in X} K(x,y) = 1 \quad \forall x \in X.$$
Then, for $$m^K_x(A):= \sum_{y \in A} K(x,y),$$
$[X, d, m^K]$ is a metric random walk space for  any  metric  $d$ on $X$. Basic Markov chain theory guarantees the existence of a unique  stationary probability measure (also called steady state) $\pi$ on $X$, that is,
$$\sum_{x \in X} \pi(x) = 1 \quad \hbox{and} \quad \pi(y) = \sum_{x \in X} \pi(x) K(x,y) \quad \quad \forall y \in X.$$
We say that $\pi$ is reversible for $K$ if the following detailed balance equation
$$K(x,y) \pi(x) = K(y,x) \pi(y)$$ holds true for $x, y \in X$.
			
\noindent (3) Consider  a locally finite weighted discrete graph $G = (V(G), E(G))$, where each edge $(x,y) \in E(G)$ (we will write $x\sim y$ if $(x,y) \in E(G)$) has a positive weight $w_{xy} = w_{yx}$ assigned. Suppose further that $w_{xy} = 0$ if $(x,y) \not\in E(G)$. The graph is equipped with the standard shortest path graph distance $d_G$, that is, $d_G(x,y)$ is the  minimal number of edges connecting $x$ and $y$.  We will assume that any two point are connected, i.e., that the graph is connected. For $x \in V(G)$ we define the weight at the vertex $x$ as
$$d_x:= \sum_{y\sim x} w_{xy} = \sum_{y\in V(G)} w_{xy},$$
and the neighbourhood $N_G(x) := \{ y \in V(G) \ : \ x\sim y\}$. The graph is locally  finite, i.e. the sets $N_G(x)$ are assumed to be finite. When all the weights are $1$, $d_x$ coincides with the degree of the vertex $x$ in a graph, that is,  the number of edges containing vertex $x$.
			
For each $x \in V(G)$  we define the following probability measure
$$m^G_x=  \frac{1}{d_x}\sum_{y \sim x} w_{xy}\,\delta_y.\\ \\
$$
We have that $[V(G), d_G, m^G]$ is a metric random walk space. It is not difficult to see that the measure $\nu_G$ defined as
$$\nu_G(A):= \sum_{x \in A} d_x,  \quad A \subset V(G)$$
is an invariant and  reversible measure for this random walk.
			
\noindent (4) From a metric measure space $(X,d, \mu)$ we can obtain a metric random walk space with the so called {\it $\epsilon$-step random walk associated to $\mu$}, as follows. Assume that balls in $X$ have finite measure and that ${\rm Supp}(\mu) = X$. Given $\epsilon > 0$, the $\epsilon$-step random walk on $X$, starting at point~$x$, consists in randomly jumping in the ball of radius $\epsilon$ around $x$, with probability proportional to $\mu$; namely
$$m^{\mu,\epsilon}_x:= \frac{\mu \res B(x, \epsilon)}{\mu(B(x, \epsilon))}.$$
Note that $\mu$ is an invariant and reversible measure for the metric random walk $[X, d, m^{\mu,\epsilon}]$.
			
\noindent (5) Given a  metric random walk  space $[X,d,m]$ with invariant and reversible measure $\nu$ for $m$, and given a $\nu$-measurable set $\Omega \subset X$ with $\nu(\Omega) > 0$, if we define, for $x\in\Omega$,
$$m^{\Omega}_x(A):=\int_A d m_x(y)+\left(\int_{X\setminus \Omega}d m_x(y)\right)\delta_x(A) \quad \forall A \subset  \Omega \ \hbox{borelian},$$
we have that $[\Omega,d,m^{\Omega}]$ is a metric random walk space and it easy to see (\cite{MST0}) that $\nu \res \Omega$ is  reversible for $m^{\Omega}$.
			
In the case that $\Omega$ is a closed bounded subset of $\R^N$, if we consider the metric random walk $[\Omega, d, m^{J,\Omega}]$, being $m^{J,\Omega} = (m^J)^{\Omega}$, that is
$$m^{J,\Omega}_x(A):=\int_A J(x-y)dy+\left(\int_{\R^n\setminus \Omega}J(x-z)dz\right)d\delta_x \quad \forall A \subset  \Omega \ \hbox{borelian},$$
}
\end{example}
	
From now on, we will assume that $[X,d,m]$ is a metric random walk space with an invariant and reversible measure $\nu$.
	
Let us recall some of the results about the $m$-Perimeter and the $m$-Total Variation given in \cite{MST1}.
	
\subsection{$m$-Perimeter}
	
In this context, the {\it $m$-interaction} between two $\nu$-measurable subsets $A$ and $B$ of $X$  is defined as
$$ L_m(A,B):= \int_A \int_B dm_x(y) d\nu(x).
$$
Whenever  $L_m(A,B) < +\infty$, by the reversibility assumption on $\nu$ with respect to $m$, we have
$$L_m(A,B)=L_m(B,A).$$
We define the concept of $m$-perimeter of a $\nu$-measurable subset $E \subset X$ as
$$P_m(E)=L_m(E,X\setminus E) = \int_E \int_{X\setminus E} dm_x(y) d\nu(x).$$
It is easy to see that
$$P_m(E) = \frac{1}{2} \int_{X}  \int_{X}  \vert \1_{E}(y) - \1_{E}(x) \vert dm_x(y) d\nu(x).
$$
Moreover,  if $E$ is $\nu$-integrable, we have
$$\displaystyle P_m(E)=\nu(E) -\int_E\int_E dm_x(y) d\nu(x).
$$

\begin{example}{\rm
(1) Let  $[\R^N, d, m^J]$ be   the metric random walk space given in Example   \ref{JJ} (1) with invariant measure $\mathcal{L}^N$. Then,
$$P_{m^J} (E) = \frac{1}{2} \int_{\R^N}  \int_{\R^N}  \vert \1_{E}(y) - \1_{E}(x) \vert J(x -y) dy dx,$$
which coincides with the concept of $J$-perimeter introduced in \cite{MRT1} (see also \cite{MRTLibro}). On the other hand,
$$P_{ m^{J,\Omega}} (E) = \frac{1}{2} \int_{\Omega}  \int_{\Omega}  \vert \1_{E}(y) - \1_{E}(x) \vert J(x -y) dy dx.$$
Note that, in general, $P_{ m^{J,\Omega}} (E) \not= P_{m^J} (E).$
			
Moreover,
$$P_{ m^{J,\Omega}} (E) = \mathcal{L}^N(E) - \int_E \int_E dm_x^{J,\Omega}(y) dx = $$
$$ = \mathcal{L}^N(E) - \int_E \int_E J(x-y) dy dx - \int_E \left( \int_{\R^N \setminus \Omega} J(x - z) dz\right) dx.$$
Therefore,
$$P_{ m^{J,\Omega}} (E) = P_{ m^{J}} (E)  - \int_E \left( \int_{\R^N \setminus \Omega} J(x - z) dz\right) dx, \quad \forall \, E \subset \Omega.$$
			
\noindent (2)
In the particular case of  a graph $[V(G), d_G, m^G ]$, given $A, B \subset V(G)$, ${\rm Cut}(A,B)$ is defined as
$${\rm Cut}(A,B):= \sum_{x \in A, y \in B} w_{xy} = L_{m^G}(A,B),$$
and the perimeter of a set $E \subset V(G)$ is given by
$$\vert \partial E \vert := {\rm Cut}(E,E^c) = \sum_{x \in E, y \in V \setminus E} w_{xy}.$$
Then, we have that
$$ \vert \partial E \vert = P_{m^G}(E) \quad \hbox{for all} \ E \subset V(G). $$
}
\end{example}
	
\subsection{$m$-Total Variation}
Associated to the  random walk $m=(m_x)$ and   $\nu$,  we define the space
$$BV_m(X):= \left\{ u :X \to \R \ \hbox{ $\nu$-measurable} \, : \, \int_{X}  \int_{X}  \vert u(y) - u(x) \vert dm_x(y) d\nu(x) < \infty \right\}.$$
We have   $L^1(X,\nu)\subset BV_m(X)$. For $u\in BV_m(X)$,  we  define its {\it $m$-total variation} as
$$TV_m(u):= \frac{1}{2} \int_{X}  \int_{X}  \vert u(y) - u(x) \vert dm_x(y) d\nu(x).$$
Note that
$$P_m(E) = TV_m(\1_E).$$
Recall the definition of the generalized product measure $\nu \otimes m_x$ (see, for instance, \cite{AFP}), which is defined as the measure in $X \times X$ given by
$$ \nu \otimes m_x(U) := \int_X   \int_X \1_{U}(x,y) dm_x(y)   d\nu(x)\quad\hbox{for }  U\in \mathcal{B}(X\times X),$$  where one needs the map $x \mapsto m_x(E)$ to be $\nu$-measurable for any Borel set $E \in \mathcal{B}(X)$.
It holds that $$\int_{X \times X} g  d(\nu \otimes m_x)   = \int_X   \int_X g(x,y) dm_x(y)  d\nu(x)$$
for every  $g\in L^1(X\times X,\nu\otimes m_x)$. Therefore, we can write
$$TV_m(u)= \frac{1}{2} \int_{X}  \int_{X}  \vert u(y) - u(x) \vert d(\nu \otimes m_x)(x,y).$$
	
\begin{example}\label{exammplee}{\rm Let $[V(G), d_G, m^G]$  be the metric random walk  given in Example \ref{JJ} (3) with the invariant and reversible  measure $\nu_G$. Then
$$TV_{m^G} (u) = \frac{1}{2} \int_{V(G)}  \int_{V(G)}  \vert u(y) - u(x) \vert dm^G_x(y) d\nu_G(x) $$ $$= \frac{1}{2} \int_{V(G)} \frac{1}{d_x} \left(\sum_{y \in V(G)} \vert u(y) - u(x) \vert w_{xy}\right) d\nu_G(x)$$ $$=  \frac{1}{2} \sum_{x \in V(G)} d_x \left(\frac{1}{d_x} \sum_{y \in V(G)} \vert u(y) - u(x) \vert w_{xy}\right) $$ $$= \frac{1}{2}  \sum_{x \in V(G)} \sum_{y \in V(G)} \vert u(y) - u(x) \vert w_{xy}.$$
Note that $TV_{m^G} (u)$ coincides with the anisotropic total variation	defined in \cite{GGOB}.}
\end{example}
	
Let us now recall some properties of the $m$-total variation given in \cite{MST1}.
	
\begin{proposition}\label{lemLipsch} If $\phi : \R \rightarrow \R$ is Lipschitz continuous  then, for every $u \in BV_m(X)$, $\phi(u) \in BV_m(X)$ and
$$TV_m(\phi(u)) \leq \Vert \phi \Vert_{Lip} TV_m(u).$$
\end{proposition}
	
\begin{proposition}\label{lemsc1}
$TV_m$ is convex and continuous in $L^1(X, \nu)$ and lower semi-continuous respect to the weak convergence in $L^q(X,\nu)$, $q = 1,2$.
\end{proposition}

As in the local and nonlocal case, we have the following coarea formula relating the total variation of a function with the perimeter of its superlevel sets.
	
\begin{theorem}[\bf Coarea formula]
For any $u \in L^1(X,\nu)$, let $E_t(u):= \{ x \in X \ : \ u(x) > t \}$. Then,
$$TV_m(u) = \int_{-\infty}^{+\infty} P_m(E_t(u))\, dt. $$
\end{theorem}
	
For a function $u : X \rightarrow \R$ we define its nonlocal gradient $\nabla u: X \times X \rightarrow \R$ as
$$\nabla u (x,y):= u(y) - u(x) \quad \forall \, x,y \in X.$$
For a function $\z : X \times X \rightarrow \R$, its {\it $m$-divergence} ${\rm div}_m \z : X \rightarrow \R$ is defined as
$$({\rm div}_m \z)(x):= \frac12 \int_{X} (\z(x,y) - \z(y,x)) dm_x(y).$$
For $p \geq 1$, we define the space
$$X_m^p(X):= \bigg\{ \z \in L^\infty(X\times X, \nu \otimes m_x) \ : \ {\rm div}_m \z \in L^p(X,\nu) \bigg\}.$$
For $u \in BV_m(X) \cap L^{p'}(X,\nu)$ and $\z \in X_m^p(X)$,  $1\le p\le \infty$, having in mind that $\nu$ is reversible, we have the following {\it Green's formula}
$$\int_{X} u(x) ({\rm div}_m \z)(x) dx = -\frac12 \int_{X \times X} \nabla u(x,y) \z(x,y)   d\nu\otimes dm_x. $$
In the next result we characterize $TV_m$ and the $m$-perimeter using the $m$-divergence operator.
Let us denote by $\hbox{sign}_0(r)$ the usual sign function and  by $\hbox{sign}(r)$ the multivalued sign function
$${\rm sign}(u)(x):=  \left\{ \begin{array}{lll} 1 \quad \quad &\hbox{if} \ \  u(x) > 0, \\ -1 \quad \quad &\hbox{if} \ \ u(x) < 0, \\ \left[-1,1\right] \quad \quad &\hbox{if} \ \ \ u = 0. \end{array}\right.$$
	
\begin{proposition} Let $1\le p\le \infty$. For $u \in BV_m(X) \cap L^{p'}(X,\nu)$, we have
$$TV_m(u) =   \sup \left\{ \int_{X} u(x) ({\rm div}_m \z)(x) d\nu(x)  \ : \ \z \in X_m^p(X), \ \Vert \z \Vert_\infty \leq 1 \right\}.$$
In particular, for any $\nu$-measurable set $E \subset X$, we have
$$P_m(E) =   \sup \left\{ \int_{E} ({\rm div}_m \z)(x) dx  \ : \ \z \in  X_m^1(X), \  \Vert \z \Vert_\infty \leq 1 \right\}.$$
\end{proposition}
	
Let us recall the concept of ergodicity (see, for example, \cite{HLL}).
	
	\begin{definition}{\rm Let $[X, d, m]$ be a metric random walk space with invariant  probability measure~$\nu$. A Borel set $B \subset X$ is said to be {\it invariant} with respect to the random walk $m$ if $m_x(B) = 1$ whenever $x$ is in $B$.
			
			The invariant  probability measure~$\nu$ is said to be {\it ergodic} if $\nu(B) = 0$ or $\nu(B) = 1$ for every invariant set $B$ with respect to the random walk $m$.
		}
	\end{definition}
	
The following result was obtained in \cite{MST0}.
\begin{theorem}\label{ergconectC}
Let $[X, d, m]$ be a metric random walk with invariant probability measure ~$\nu$. Then, the following assertions are equivalent:
\begin{itemize}
\item[(i)] $\nu$ is ergodic.
\item[(ii)] If $A,B\subset X$ are disjoint sets such that $0<\nu(A),\nu(B)<\nu(X)$ and $A \cup B = X$, then $L_m(A,B)\neq 0$.
\end{itemize}
\end{theorem}

 In particular, the Theorem above gives an important family of ergodic measures: on a length space, a doubling measure $\nu$ is ergodic with respect to the random walk $m^{\nu,\varepsilon}$.
	
\subsection{BV-functions in metric measure spaces}

Let $(X, d, \nu)$ be a metric measure space. For functions in $L^1(X, \nu)$ we have the concept of total variation introduced by Miranda in \cite{Miranda1} (see also  \cite{ADiM}).  To introduce this concept recall that for a function $u : X \rightarrow \R$, its {\it slope} (also called local Lipschitz constant) is defined by
$$\vert \nabla u \vert(x) := \limsup_{y \to x} \frac{\vert u(y) - u(x)\vert}{d(x,y)},$$
with the convention that  $\vert \nabla u \vert(x) = 0$ if $x$ is an isolated point.

A function $u \in L^1(X, \nu)$ is said to be a {\it BV-function} if there exist locally Lipschitz functions $u_n$ converging to $u$ in $L^1(X, \nu)$ and such that
$$\sup_{n \in \N} \int_X \vert \nabla u_n \vert d\nu(x) < \infty.$$
We shall denote the space of all BV-functions by $BV(X,d, \nu)$. For $u \in BV(X,d, \nu)$ the total variation $\vert D u \vert_{\nu}$ on an open set $A \subset X$ is defined as:
$$
\vert D u \vert_{\nu}(A):= \inf \left\{ \liminf_{n \to \infty} \int_A \vert \nabla u_n \vert(x) d \nu(x) \ : \ u_n \in Lip_{loc}(X), \ u_n \to u \ \hbox{in} \ L^1(A, \nu) \right\}.
$$

We say that a measure $\nu$ on a metric space $X$ is doubling, if there exists a constant $C_d \geq 1$ such that following condition holds:
$$
0 < \nu(B(x,2r)) \leq C_d \, \nu(B(x, r)) < \infty
$$
for all $x \in X$ and $r > 0$. The constant $C_d$ is called the doubling constant of $X$. By iterating the doubling condition, we see that a classical estimate holds (see for instance \cite{Haj}):
		
\begin{proposition}\label{prop:doubling}
Let $y \in B(x,R)$ and set $r \in (0,R)$. Then
$$ \frac{\nu(B(y,r))}{\nu(B(x,R))} \geq 4^{-s} \bigg( \frac{r}{R} \bigg)^s$$
for all $s \geq \log_2(C_d)$. The number $\log_2(C_d)$ is called the homogenous dimension of $X$.
\end{proposition}

The metric measure space $(X,d, \nu)$ is said to {\it support a local $1$-Poincar\'{e} inequality} if there exist constants $c>0$ and $\lambda \geq1$ such that, for any $u \in {\rm Lip}(X,d)$, the inequality
$$\int_{B(x,r)} \vert u(y) - u_{B(x,r)} \vert d\nu(y) \leq c r \int_{ B(x,\lambda r)} \vert \nabla u \vert (y) d \nu(y)$$
holds, where
$$u_{B(x,r)}:= \frac{1}{\nu(B(x,r))}\int_{ B(x,r) } u(y) d\nu(y).$$
		
We recall the following result proved in \cite[Theorem 2.22]{MST1} (see also \cite[Theorem 3.1]{MMS}), which links the total variation defined above and the nonlocal total variation of the type presented in Example \ref{JJ}(4):
		
\begin{theorem}\label{thm:connectionwithlocal}
Let $(X,d,\nu)$ be a metric measure space with doubling measure $\nu$ and supporting a local 1-Poincar\'e inequality.
Given $u \in L^1(X,\nu)$, we have that $u \in BV(X,d,\nu)$ if and only if
$$ \liminf_{\varepsilon \rightarrow 0} \frac{1}{\varepsilon} TV_{m^{\nu,\varepsilon}}(u) < \infty.$$
\end{theorem}
	
\section{Least Gradient Functions in  Metric Random Walk Spaces}
	
From now on, we will assume that $[X,d,m]$ is a metric random walk space with an invariant and reversible measure $\nu$, with $\nu(X) < \infty$.
	
Given $\Omega \subset X$ a bounded $\nu$-measurable set, we define its $m$-boundary as
$$\partial_m \Omega := \{ x \in X \setminus \Omega \ : \ m_x(\Omega) > 0 \}.$$
We set $\Omega_m:= \Omega \cup \partial_m \Omega$.
	
\subsection{Nonlocal least gradient problem}\label{sec:lgpnonlocal}
We will deal with $\Omega \subset X$ such that $0<\nu(\Omega)<\nu(X)$. Then, assuming that $\nu$ is ergodic, by Theorem~\ref{ergconectC}, we have that $$\nu( \partial_m \Omega) >0.$$   From now on we will assume that $\nu$ is ergodic.
	
Given a function $\psi \in L^1(\partial_m \Omega)$ we consider {\it $m$-least gradient problem} as the variational problem
\begin{equation}\label{least1}
\min \{ TV_m (u) \ : \ u \in BV_m(\Omega_m), \ \hbox{such that} \ u=\psi \ \nu-\hbox{a.e. on} \ \partial_m \Omega \}.
\end{equation}

 We may equivalently state the problem on the whole space $X$ and not only on $\Omega_m$; suppose that instead of defining the boundary condition only on $\partial_m \Omega$, we take $\psi \in L^1(X,\nu)$ and we extend $u$ to the whole space $X$ by setting $u = \psi$ $\nu$-a.e. on $X \backslash \Omega$. Let us denote by $TV_m(u,A)$ the total variation of $u \res A$ in a Borel set $A \subset X$. By the reversibility of $\nu$ and the definition of $\partial_m \Omega$, we have
\begin{equation}
TV_m(u,X) = TV_m (u, \Omega_m) + \int_{\partial_m\Omega}\int_{X \backslash \Omega_m} |\psi(y) - \psi(x)| \, dm_x(y)d\nu(x) + TV_m(\psi, X \backslash \Omega_m),
\end{equation}
hence the total variations of $u$ on $X$ and on $\Omega_m$ differ by a constant, so they have the same minimizers.

Using the direct method of calculus of variation we have the following existence result.
	
\begin{theorem}
If $\psi \in L^\infty(\partial_m \Omega)$ the $m$-least gradient problem \eqref{least1} has a solution.
\end{theorem}
\begin{proof} Let
		$$\tau = \inf \{ TV_m (u) \ : \ u \in BV_m(\Omega_m), \ \hbox{such that} \ u=\psi \ \hbox{on} \ \partial_m \Omega \},$$
		and $u_n \in  BV_m(\Omega_m)$ a sequence of admissible functions such that $TV_m(u_n) \to \tau$. Firstly, we correct the functions $u_n$, we set
		$$\widetilde{u_n}(x):= \left\{ \begin{array}{lll} \Vert \psi \Vert_\infty \quad &\hbox{if} \ \ u_n(x) > \Vert \psi \Vert_\infty  \\[10pt] u_n(x) \quad &\hbox{if} \ \ u_n(x) \in [- \Vert \psi \Vert_\infty, \Vert \psi \Vert_\infty] \\[10pt]- \Vert \psi \Vert_\infty \quad &\hbox{if} \ \ u_n(x) < -\Vert \psi \Vert_\infty.\end{array}  \right.$$
		By Proposition \ref{lemLipsch}, we have $TV_m(\widetilde{u_n}) \leq TV_m(u_n)$, hence the functions $\widetilde{u_n}$ are admissible. Moreover, the sequence $\{\widetilde{u_n} \}$ is uniformly bounded, then by the Dunford-Pettis Theorem, we can assume, taking a subsequence if necessary, that $\widetilde{u_n} \rightharpoonup u \in BV_m(\Omega_m)$ weakly in $L^1(X, \nu)$. Then, by Proposition \ref{lemsc1}, we have
		$$TV_m(u) \leq \liminf_{n \to \infty} TV_m(\widetilde{u_n}) = \tau.$$
		Finally, since $u$ is the weak limit of the sequence $\{\widetilde{u_n} \}$, we have
		$$\liminf_{n \to \infty} \widetilde{u_n}(x) \leq u(x) \leq \limsup_{n \to \infty} \widetilde{u_n}(x) \quad \nu-a.e. \ x \in \Omega_m.$$
		Thus, $u = \psi$ $\nu-\hbox{a.e. on} \ \partial_m \Omega$, so $u$ is a solution of the $m$-least gradient problem \eqref{least1}.
	\end{proof}
	
	Now, we define $m-$least gradient functions and prove nonlocal analogues of Miranda's Theorem, see \cite{Mir}, and a theorem by Bombieri, de Giorgi and Giusti linking a $m-$least gradient function to the minimality of its superlevel sets, see \cite{BdGG}.
	
	\begin{definition}
		We say that $u \in BV_m(X)$ is a $m-$least gradient function on $\Omega$, if for every $g \in BV_m(X)$ such that $g \equiv 0$ $\nu$-a.e. on $X \backslash \Omega$ we have
		$$ TV_m(u) \leq TV_m(u + g).$$
	\end{definition}
	
	\begin{proposition}\label{prop:miranda}
		Suppose that $u_n \in BV_m(X)$ is a sequence of $m-$least gradient functions in $\Omega$ convergent in $L^1(X, \nu)$ to $u \in BV_m(X)$. Then $u$ is a function of $m-$least gradient in $\Omega$.
	\end{proposition}
	
	\begin{proof}
		By Proposition \ref{lemsc1}, $m-$total variation is continuous with respect to convergence in $L^1(X, \nu)$. Fix $g \in BV_m(X)$ with support in $\Omega$. Hence
		$$ TV_m(u) \leftarrow TV_m(u_n) \leq TV_m(u_n + g) \rightarrow TV_m(u + g), $$
		hence $u$ is a function of $m-$least gradient.
	\end{proof}
	
	\begin{theorem}
		Suppose that $u \in BV_m(X)$. Then $u$ is a function of $m-$least gradient in $\Omega$ if and only if $\1_{E_t(u)}$ is a function of $m-$least gradient for all (equivalently - almost all) $t \in \mathbb{R}$.
	\end{theorem}
	
\begin{proof}
Fix $t > 0$ and let $u_1 = \min(u, t)$ and $u_2 = u - u_1 = \max(u - t, 0)$. By the coarea formula, we have
$$ TV_m(u) = \int_{-\infty}^{+\infty} P_m(E_t(u)) dt = \int_{-\infty}^{t} P_m(E_t(u)) dt + \int_{t}^{+\infty} P_m(E_t(u)) dt = $$
$$ = TV_m(u_1) + TV_m(u_2).$$
This implies that $u_1$ and $u_2$ are functions of $m-$least gradient in $\Omega$. In fact, for any $g \in BV_m(X)$ such that $g \equiv 0$ $\nu$-a.e. on $X \backslash \Omega$,  we have
$$ TV_m(u_1) + TV_m(u_2) = TV_m(u) \leq TV_m(u + g) \leq TV_m(u_1 + g) + TV_m(u_2),$$
hence $u_1$ is a function of $m-$least gradient. We make a similar argument for $u_2$. Now, we set
$$ u_{\varepsilon, t} = \frac{1}{\varepsilon} \min(\varepsilon, \max(u - t, 0)) $$
and notice that by the above argument it is a function of $m-$least gradient (it is a rescaled minimum of $u_2$ and a constant).
		
Now, assume that $\nu(\{ u = t \}) = 0$. Then $u_{\varepsilon, t} \rightarrow \1_{E_t(u)}$ in $L^1(X, \nu)$. Indeed,
$$ \int_X |u_{\varepsilon, t} - \1_{E_t(u)}| d\nu = \int_{X \backslash E_t(u)} |u_{\varepsilon, t} - \1_{E_t(u)}| d\nu + \int_{E_t(u) \backslash E_{t + \varepsilon} (u)} |u_{\varepsilon, t} - \1_{E_t(u)}| d\nu + $$
$$ + \int_{E_{t + \varepsilon} (u)} |u_{\varepsilon, t} - \1_{E_t(u)}| d\nu = \int_{E_t(u) \backslash E_{t + \varepsilon} (u)} |u_{\varepsilon, t} - \1_{E_t(u)}| d\nu,$$
because the first and last integrals are equal to zero. But
$$\int_{E_t(u) \backslash E_{t + \varepsilon} (u)} |u_{\varepsilon, t} - \1_{E_t(u)}| d\nu \leq \nu(E_t(u) \backslash E_{t + \varepsilon} (u)),$$
which goes to zero if $\nu(\{ u = t \}) = 0$. By Proposition \ref{prop:miranda} $\1_{E_t(u)}$ is a function of $m-$least gradient.
		
Now, assume that $\nu(\{ u = t \}) > 0$. As $E_t = \bigcap_{t' > t} E_{t'}$, there exists a sequence $t_n \rightarrow t$ such that $\nu(\{ u = t_n \}) = 0$ and that $\1_{E_{t_n}} \rightarrow \1_{E_t}$ in $L^1(X, \nu)$. For each $n$ the function $\1_{E_{t_n}}$ is a function of $m-$least gradient; hence, by Proposition \ref{prop:miranda} $\1_{E_t(u)}$ is a function of $m-$least gradient.
		
The implication in the other direction follows directly from the coarea formula.
\end{proof}

	Let $\Omega \subset \R^N$ be a bounded open set. In \cite{MazRoSe} it is proved that the Dirichlet problem for the $1$--Laplacian operator
	\begin{equation}\label{ELe}
	\left\{\begin{array}{ll}
	\displaystyle -\mbox{div}\Big(\frac{Du}{|Du|}\Big)=0\,,&\hbox{in }\Omega\,,\\[10pt]
	u=h\,,&\hbox{on }\partial\Omega\,,
	\end{array}\right.
	\end{equation}
	has a solution $u \in BV(\Omega)$  for every $h\in
	L^1(\partial\Omega)$. The relaxed energy functional associated to problem \eqref{ELe} is the functional
	$\Phi_h: L^{\frac N{N-1}}(\Omega)
	\rightarrow (-\infty,+\infty]$ defined by
	\begin{equation}
	\label{Def_Phi_PDE}
	\Phi_h(u) = \left\{ \begin{array}{ll} \displaystyle
	\int_{\Omega}\vert Du \vert + \int_{\partial \Omega} \vert u - h \vert\, d \mathcal H^{N-1} & \hbox{if} \ u \in BV(\Omega),
	\\[10pt] +\infty & \hbox{if} \  u \in  L^{\frac N{N-1}}(\Omega) \setminus BV(\Omega).
	\end{array}\right.
	\end{equation}
	In \cite{MazRoSe} it is showed that the minimizer of the functional $\Phi_h$ coincides with the solution of problem~\eqref{ELe} and with the function of least gradient on $\Omega$ that coincides with $h$ on the boundary. A nonlocal version of  the above first statement  was obtained in \cite{MazPeRoTo}, more precisely
	for the random walk $m^J$  given in Example~\ref{JJ}(1). We will see now that using an adaptation of the method developed in \cite{MazPeRoTo}, we can obtain similar results for general metric random walk spaces. Although some of the proofs are similar, we will give them for the sake of completeness.

	Given a function $\psi \in L^1(\partial_m\Omega)$ and $u \in BV_m(\Omega)$, we define the function
	$$u_\psi(x):= \left\{\begin{array}{ll} u(x) \quad &\hbox{if} \ \  x \in \Omega \\[10pt] \psi(x) \quad &\hbox{if} \ \ x \in \partial_m\Omega. \end{array} \right.$$
	Consider the relaxed energy functional $\mathcal{J}_{\psi} : L^1(\Omega,\nu)
	\rightarrow [0, + \infty[$ given by
	\begin{equation}\label{nllgrad.intro}
	\mathcal{J}_{\psi} (u):= TV_m(u_{\psi}) = \frac{1}{2}\int_{\Omega_m}\int_{\Omega_m}
	| u_{\psi}(y) - u_{\psi}(x) | \, dm_x(y)d\nu(x).\end{equation}
	This functional $\mathcal{J}_{\psi}$ is the nonlocal version of the energy functional $\Phi_h$ defined by~\eqref{Def_Phi_PDE}.

Since if $v \in BV_m(\Omega_m)$, such that $v=\psi$ \ $\nu-\hbox{a.e. on} \ \partial_m \Omega$, then $(v_{\vert \Omega})_\psi = v$, we have the following result.
\begin{proposition}\label{equal1} Given $\psi \in L^1(\partial_m\Omega)$, For $u \in BV_m(\Omega_m)$, such that $u=\psi$ \ $\nu-\hbox{a.e. on} \ \partial_m \Omega$, the following are equivalent:
\begin{itemize}
\item[(i)] $u$ is a solution of the $m$-least gradient problem \eqref{least1}.
\item[(ii)] $u_{\vert \Omega}$ is a minimizer of $\mathcal{J}_{\psi}$.
\end{itemize}
\end{proposition}

	In \cite{MST1}, the authors study the $m$-1-Laplacian operator $\Delta^m_1$, which formally is the operator
	$$\Delta^m_1 u(x) = \int_{\Omega_m}
	\frac{u_\psi(y)-u(x)}{|u_\psi(y)-u(x)|}dm_x(y) \quad \mbox{for } x \in \Omega.$$
	Consider the {\it nonlocal
		$1$-Laplace problem} with Dirichlet boundary condition~$\psi$:
	\begin{equation}\label{1700.intro}
	\left\{\begin{array} {ll} \displaystyle  -\Delta^m_1 u(x)=0,& x\in\Omega,
	\\[10pt] u(x)=\psi(x),& x\in \partial_m\Omega.
	\end{array}\right.
	\end{equation}
	
	\begin{definition}\label{Defi.1.var} {\rm Let $\psi\in L^1(\partial_m\Omega)$.
			We say that $u\in BV_m(\Omega)$ is a    solution  to
			\eqref{1700.intro}  if there exists $g\in L^\infty(\Omega_m\times \Omega_m)$ with
			$\|g\|_\infty\leq 1$ verifying
			\begin{equation}\label{lasim}
			g(x,y)=-g(y,x)\quad\hbox{for  $(\nu\otimes dm_x)$-a.e $(x,y)$ \ in $\Omega_m\times \Omega_m$},
			\end{equation}
			\begin{equation}\label{1-lapla2.var}
			g(x,y)\in \mbox{sign}(u_\psi(y)-u_\psi(x))\quad
			\hbox{for  $(\nu\otimes dm_x)$-a.e $(x,y)$ \ in $\Omega_m\times \Omega_m$},
			\end{equation} and
			\begin{equation}\label{1-lapla.var}
			-\int_{\Omega_m}g(x,y)\,dm_x(y)= 0 \quad \mbox{for }\nu-\mbox{a.e }x\in \Omega.
			\end{equation}
		}
	\end{definition}
	
	\begin{definition}  Let $q \geq 1$. We say that $(m,\nu)$ satisfies a {\it $q$-Poincar\'{e} Inequality in $\Omega$} if there exists $\lambda > 0$ such that
		$$
		\lambda\int_\Omega \left|u(x) \right|^q \, d\nu(x)\le \int_{\Omega}
		\int_{\Omega_m}  |u_{\psi}(y)-u(x)|^q \, dm_x(y) \,
		d\nu(x)+\int_{\partial_m\Omega}|\psi (y)|^q \, d\nu(y)
		$$
		for all $u\in L^q(\Omega, \nu)$ and  $\psi \in L^q(\partial_m \Omega)$.
	\end{definition}
	
	We have the following result.
	\begin{theorem} \label{teo.varia}   Let $\psi\in L^\infty(\partial_m\Omega)$.  If $(m,\nu)$ satisfies a $p$-Poincar\'{e} inequality in $\Omega$ for all $p > 1$,
		then $u\in L^1(\Omega)$ is a solution to problem \eqref{1700.intro} if and only if it is a minimizer of the functional $\mathcal{J}_{\psi}$ given by the formula~\eqref{nllgrad.intro}.
	\end{theorem}
	In fact, it is enough to assume that there exist solutions to \eqref{1700.intro}; hence, we will start by proving the existence of  solutions  to
	\eqref{1700.intro}. We prove this existence result by taking limits as $p$ goes to $1$ in the following  Dirichlet problem
	\begin{equation}\label{pLPLc}\left\{\begin{array}{ll}
	\D - \int_{\Omega_m}|(u_p)_\psi(y)-u_p(x)|^{p-2}((u_p)_\psi(y)-u_p(x))dm_x(y)= 0,
	&x\in\Omega,\\[10pt] u_p=\psi,& x\in \partial_m\Omega.
	\end{array}\right.\end{equation}
	
	\begin{theorem}\label{theo1} Given $\psi\in
		L^\infty(\partial_m\Omega)$ and $p >1$.  If $(m,\nu)$ satisfies a $p$-Poincar\'{e} inequality in $\Omega$, then  there exists a
		solution $u_p$  of problem \eqref{pLPLc}. Moreover,
		\begin{equation}\label{MaxP}\|u_p\|_\infty\leq \Vert\psi\Vert_\infty \,.\end{equation}
	\end{theorem}

	\begin{proof}
		Let us consider the functional
		$$\mathcal{F}_p(u):= \frac{1}{2p} \int_{\Omega_m}\int_{\Omega_m}|u_\psi(y)-u_\psi(x)|^{p}dm_x(y)d\nu(x),
		\quad u \in L^p(\Omega).$$
		Set $$\theta := \inf_{u \in L^p(\Omega,\nu)} \mathcal{F}_p(u),$$ and let $\{ u_n \}$ be a minimizing sequence. Then,
		$$\theta = \lim_{n \to \infty} \mathcal{F}_p(u_n)  \quad \hbox{and} \quad K:= \sup_{n \in \NN} \mathcal{F}_p(u_n)  < + \infty\,.$$
		Since $(m,\nu)$ satisfies a $p$-Poincar\'{e} inequality,
		$$ \lambda\int_\Omega \left|u_n(x) \right|^p \, dx\le  \int_{\Omega}
		\int_{\Omega_m}  |u_{\psi}(y)-u(x)|^q \, dm_xy \,
		d\nu(x)x+\int_{\partial_m\Omega}|\psi (y)|^q \, d\nu(y)$$ $$= 2p \mathcal{F}_p(u_n)  +\int_{\partial_m\Omega}|\psi (y)|^p \,
		dy\leq 2pK+\int_{\partial_m\Omega}|\psi (y)|^q \, d\nu(y). $$
		Therefore,  we obtain that
		$$\int_{\Omega} |u_n(x)|^p \, dx \leq C \quad \forall n \in \NN.$$
		Hence, up to a subsequence, we have
		$$u_n \rightharpoonup u_p \quad \hbox{in} \ \  L^p(\Omega).$$
		Furthermore, using the weak lower semi-continuity of the functional $\mathcal{F}_p$, we get
		$$\mathcal{F}_p(u_p) = \inf_{u \in L^p(\Omega)} \mathcal{F}_p(u).$$
		Thus, given $\lambda >0$ and $w \in L^p(\Omega)$ (we extend it to $\partial_m\Omega$ by zero),
		we have
		$$0 \leq \frac{\mathcal{F}_p(u_p + \lambda w)- \mathcal{F}_p(u_p)}{\lambda}, $$
		or equivalently,
		{\small $$0 \leq \frac12\int_{\Omega_m}\int_{\Omega_m}
			\left[\frac{|(u_p)_{\psi}(y) + \lambda w_{\psi}(y)-((u_p)_{\psi}(x) + \lambda w_{\psi}(x))|^{p} - |(u_p)_{\psi}(y)-(u_p)_{\psi}(x)|^{p} }{p\lambda}\right] \qquad\qquad$$
			$$\qquad\qquad\qquad\qquad\qquad\qquad\qquad\qquad\qquad\qquad\qquad\qquad\qquad\qquad\qquad\qquad\qquad dm_x(y)d\nu(x).$$}
		Now, since $p >1$, we pass  to the limit as $\lambda \downarrow 0$ to deduce
		$$0\leq \frac{1}{2}\int_{\Omega_m}\int_{\Omega_m} |(u_p)_{\psi}(y)-(u_p)_{\psi}(x)|^{p-2}((u_p)_{\psi}(y)-(u_p)_{\psi}(x)) \times \qquad \qquad \qquad \qquad \qquad $$
		$$ \qquad \qquad\qquad \qquad \qquad \qquad \qquad \qquad \qquad\qquad\qquad \times ((w)_{\psi}(y)-(w)_{\psi}(x)) \, dm_x(y) \, d\nu(x).$$
		Taking $\lambda <0$ and proceeding as above we obtain the reverse inequality. Consequently, we conclude that
		$$0= \frac{1}{2}\int_{\Omega_m}\int_{\Omega_m} |(u_p)_{\psi}(y)-(u_p)_{\psi}(x)|^{p-2}((u_p)_{\psi}(y)-(u_p)_{\psi}(x)) \times \qquad\qquad\qquad\qquad\qquad\qquad\qquad\qquad\qquad $$
		$$ \qquad\qquad\qquad\qquad\qquad\qquad\qquad\qquad\qquad\qquad \times ((w)_{\psi}(y)-(w)_{\psi}(x)) dm_x(y)d\nu(x) = $$
		$$ = - \int_{\Omega_m}\int_{\Omega_m} |(u_p)_{\psi}(y)-(u_p)_{\psi}(x)|^{p-2}((u_p)_{\psi}(y)-(u_p)_{\psi}(x)) dm_x(y) (w)_{\psi}(x) d\nu(x).$$
		In particular, since $w=0$ in $\partial_m\Omega$, it follows that
		$$ 0=  -\int_{\Omega}\int_{\Omega_m}|(u_p)_{\psi}(y)- u_p (x)|^{p-2}((u_p)_{\psi}(y)- u_p (x)) dm_x(y) w(x) d\nu(x),$$
		which shows that $u_p$ is a solution of \eqref{pLPLc}. To finish let us see that \eqref{MaxP} holds.
		
		Set $M:=\Vert\psi\Vert_\infty$, multiply   equation  \eqref{pLPLc} by
		$(u_p-M)^+$ and integrate over $\Omega$ to obtain
		$$0 = - \int_{\Omega_m}\int_{\Omega_m}|(u_p)_{\psi}(y)-(u_p)_{\psi}(x)|^{p-2}
		((u_p)_{\psi}(y)-(u_p)_{\psi}(x))dm_x(y)((u_p)_{\psi}-M)^+(x)d\nu(x),$$
		or equivalently
			$$0=\frac{1}{2}\int_{\Omega_m}\int_{\Omega_m}\left(|(u_p)_{\psi}(y)-(u_p)_{\psi}(x)|^{p-2}((u_p)_{\psi}(y)-(u_p)_{\psi}(x))\right) \times \qquad\qquad \qquad\qquad
			$$ $$\qquad\qquad \qquad\qquad \qquad\qquad  \times \left(((u_p)_{\psi}-M)^+(y)-((u_p)_{\psi}-M)^+(x)\right)dm_x(y)d\nu(x).$$
		In addition, since
		$$\vert r - s\vert^{p-2} (r - s) (r^+ - s^+) \geq \vert r^+ - s^+\vert^{p},$$
		it holds that
		$$\int_{\Omega_m}\int_{\Omega_m}\left|((u_p)_{\psi}-M)^+(y)-((u_p)_{\psi}-M)^+(x)\right|^{p}dm_x(y)d\nu(x) \leq 0.$$
		Then, using again the $p$-Poincar\'{e} inequality we get
		$$\int_{\Omega} \vert (u_p-M)^+(x) \vert^p \, d\nu(x) = 0.$$
		This shows that $u_p \leq M$ $\nu$-a.e. in $\Omega$ for any $p> 1$.
		Analogously, we can verify that $u_p \geq -M$ $\nu$-a.e. in $\Omega$. Thus $\label{acobb}\Vert u_p\Vert_\infty\le M$ for every $p > 1$.
	\end{proof}
	
We are now ready to prove the existence of solutions to problem	\eqref{1700.intro}.
	
\begin{theorem}\label{theo1.intro}
Given $\psi\in L^\infty(\partial_m\Omega)$.  If $(m,\nu)$ satisfies a $p$-Poincar\'{e} inequality in $\Omega$ for all $p > 1$, then there exists a  solution to problem~\eqref{1700.intro}.
\end{theorem}
	
\begin{proof}
By Theorem \ref{theo1}, there exists a subsequence $p_n\to 1$, still denoted by $p$, such that
$$u_p \rightharpoonup u\quad\hbox{weakly in  } L^1(\Omega)$$
and
$$|(u_p)_{\psi}(y)-(u_p)_{\psi}(x)|^{p-2}((u_p)_{\psi}(y)-(u_p)_{\psi}(x)) \rightharpoonup g(x,y) \quad\hbox{weakly in } L^1(\Omega_m\times\Omega_m). $$
The function $g$ is $L^\infty$-bounded by $1$, satisfies
$$	-\int_{\Omega_m}g(x,y)\,dm_x(y)=0 \quad \mbox{for }\nu-\mbox{a.e }x\in \Omega,$$
and, moreover, it is antisymmetric. In order to see that
$$g(x,y)\in \mbox{sign}(u_\psi(y)-u_\psi(x))\quad \hbox{for  $(\nu\otimes dm_x)$-a.e $(x,y)$ \ in $\Omega_m\times \Omega_m$},$$
we need to prove that
\begin{equation}\label{HHg}
-\int_{\Omega_m}\int_{\Omega_m}g(x,y)\,dm_x(y) u_\psi(x)\,d\nu(x)=\frac12\int_{\Omega_m}\int_{\Omega_m}|u_\psi(y)-u_\psi(x)|\,dm_x(y)d\nu(x).
\end{equation}
In fact, it holds that
$$
\begin{array}{l}
\displaystyle
\frac12 \int_{\Omega_m}\int_{\Omega_m} |(u_p)_{\psi}(y)-(u_p)_{\psi}(x)|^{p}\, dm_x(y)d\nu(x) = \\[10pt]
\qquad \displaystyle =-\int_{\Omega_m}\int_{\Omega_m}|(u_p)_{\psi}(y)-(u_p)_{\psi}(x)|^{p-2}((u_p)_{\psi}(y)-(u_p)_{\psi}(x))\, dm_x(y)
(u_p)_{\psi}(x)d\nu(x) =
\\[10pt]
\qquad \displaystyle
=-\int_{\partial_m\Omega}\int_{\Omega_m} |u_p(y)-u_p(x)|^{p-2}(u_p(y)-u_p(x))\, dm_x(y) \psi(x)\,d\nu(x).
\end{array}
$$
Therefore,
\begin{equation}\label{otra}
		\begin{array}{l}
		\displaystyle
		\lim_{p \rightarrow 1}\frac12 \int_{\Omega_m}\int_{\Omega_m}|u_p(y)-u_p(x)|^{p}\,
		dm_x(y)d\nu(x)  \displaystyle
		=
		-\int_{\partial_m\Omega}\int_{\Omega_m}g(x,y)\, dm_x(y)
		\psi(x)\,d\nu(x) =
		\\[10pt]
		\qquad \displaystyle
		=-\int_{\Omega_m}\int_{\Omega_m}g(x,y)\, dm_x(y)
		u_\psi(x)\,d\nu(x).
		\end{array}
\end{equation}
Now, by monotonicity it easy to see that, for all $\rho \in L^{\infty}(\Omega)$, $$
		\begin{array}{l}
		\displaystyle
		- \int_{\Omega_m} \int_{\Omega_m}
		\vert \rho_\psi(y)-\rho_\psi(x)\vert^{p-2}(\rho_\psi(y) - \rho_\psi(x)) \,
		dm_x(y)\, ((u_{p})_\psi(x)-\rho_\psi(x))\, d\nu(x) \le \\[10pt]
		\displaystyle \le  - \int_{\Omega_m} \int_{\Omega_m} \vert
		(u_{p})_\psi(y)-(u_{p})_\psi(x)\vert^{p-2} \times \\
		\displaystyle \qquad \qquad \qquad \qquad \qquad \qquad \times ((u_{p})_\psi(y) - (u_{p})_\psi(x)) \, dm_x(y) ((u_{p})_\psi(x)-\rho_\psi(x)) \, d\nu(x).
		\end{array}
		$$
		Taking limits as $p \to 1$ and invoking \eqref{otra} we get
		$$
		\begin{array}{l}
		\displaystyle -\int_{\Omega_m}\int_{\Omega_m} \,
		\hbox{sign}_0(\rho_\psi(y)-\rho_\psi(x))\,
		dm_x(y)\, (u_\psi(x)-\rho_\psi(x))\, d\nu(x) \le \\[10pt]
		\qquad \displaystyle \le -\int_{\Omega_m}\int_{\Omega_m}
		g(x,y)\, dm_x(y)\, (u_\psi(x)-\rho_\psi(x))\, d\nu(x),
		\end{array}
		$$
		where
		$$
		\mbox{sign}_0(z)=\left\{\begin{array}{lll}1\quad&\mbox{if }z>0,\\[0.1cm]
		0&\mbox{if }z=0,\\[0.1cm]
		-1&\mbox{if }z<0.
		\end{array}\right.
		$$
		Taking $\rho=u\pm\lambda u$, $\lambda>0$, dividing by $\lambda$, and letting
		$\lambda \to 0$, we obtain \eqref{HHg}, which finishes the proof.
	\end{proof}

	\begin{proof}[Proof of Theorem~\ref{teo.varia}]  Let $u$ be a solution of problem
		\eqref{1700.intro}. Then, there exists
		$g\in L^\infty(\Omega_m\times
		\Omega_m)$ with $\|g\|_\infty\leq 1$ verifying   \eqref{lasim},
		\eqref{1-lapla2.var} and \eqref{1-lapla.var}.

		Given $w \in L^1(\Omega, \nu)$, multiplying \eqref{1-lapla.var} by
		$w(x) -u(x)$, integrating,  and having in mind
		\eqref{1-lapla2.var} and the antisymmetry of $g$, \eqref{lasim}, we
		get
		$$
		\begin{array}{l}
		\displaystyle
		0= -\int_{\Omega_m}\int_{\Omega_m}g(x,y)dm_x(y) (w_{\psi}(x) -u_\psi(x))d\nu(x) \\[10pt]
		\quad \displaystyle
		= \frac12 \int_{\Omega_m}\int_{\Omega_m}g(x,y)[(w_\psi(y)-w_\psi(x)) -
		(u_\psi(y)-u_\psi(x))]dm_x(y)d\nu(x)
		\\[10pt]
		\quad \displaystyle
		\leq
		\frac12\int_{\Omega_m}\int_{\Omega_m}|w_\psi(y)-w_\psi(x)|dm_x(y)d\nu(x)
		-
		\frac12\int_{\Omega_m}\int_{\Omega_m}|u_\psi(y)-u_\psi(x)|dm_x(y)d\nu(x)
		\\[10pt]
		\quad \displaystyle
		= \mathcal{J}_{\psi} (w) - \mathcal{J}_{\psi} (u).
		\end{array}
		$$
		Therefore,
		$u$ is a minimizer of $\mathcal{J}_{\psi}$.
		
		Assume now that $u$ minimizes the functional  $ \mathcal{J}_{\psi}$.   Theorem~\ref{theo1.intro} shows the existence of a solution $\overline{u}$ of
		\eqref{1700.intro}. Namely, there exists $g:\Omega_m\times
		\Omega_m\to \RR$ such that $g\in L^\infty(\Omega_m\times
		\Omega_m)$, $\|g\|_\infty\leq 1$, $g(x,y)=-g(y,x)$ $(\nu\otimes dm_x)$-a.e $(x,y)$
		in $\Omega_m\times \Omega_m$,
		\begin{equation}\label{1-lapla22}
		g(x,y)\in
		\mbox{sign}(\overline{u}_\psi(y)-\overline{u}_\psi(x))\quad
		(\nu\otimes dm_x)-\mbox{a.e }(x,y)\in \Omega_m\times \Omega_m,
		\end{equation}
		and \begin{equation}\label{1-lapla2456}
		-\int_{\Omega_m}g(x,y)\,dy= 0 \quad \nu-\mbox{a.e }x\in \Omega.
		\end{equation}
		Since $u$ is a minimizer of $ \mathcal{J}_{\psi}$, we have
		$$ \mathcal{J}_{\psi} (\overline{u}) - \mathcal{J}_{\psi} (u)= 0.$$
		
		On the other hand, arguing as in the other implication, we
		obtain that
		$$
		\begin{array}{l}
		\displaystyle
		0= -\int_{\Omega_m}\int_{\Omega_m}g(x,y)\,dm_x(y)(\overline{u}_\psi(x)-u_\psi(x))d\nu(x)
		\\[10pt]
		\quad \displaystyle
		=\frac12 \int_{\Omega_m}\int_{\Omega_m}g(x,y)[(\overline{u}_{\psi}(y)-\overline{u}_{\psi}(x)) -
		(u_\psi(y)-u_\psi(x))]dm_x(y)d\nu(x)
		\\[10pt]
		\quad \displaystyle
		=
		\frac12\int_{\Omega_m}\int_{\Omega_m}|\overline{u}_\psi(y)-\overline{u}_\psi(x)|dm_x(y)d\nu(x)
\\[10pt]
		\quad \displaystyle
		-
		\frac12\int_{\Omega_m}\int_{\Omega_m}g(x,y)(u_\psi(y)-u_\psi(x))dm_x(y)d\nu(x)
		\\[10pt]
		\quad \displaystyle
		=\mathcal{J}_{\psi} (\overline{u})-\frac12\int_{\Omega_m}\int_{\Omega_m}g(x,y)(u_\psi(y)-u_\psi(x))dm_x(y)d\nu(x).
		\end{array}
		$$
		Therefore,
		$$\frac12\int_{\Omega_m}\int_{\Omega_m}g(x,y)(u_\psi(y)-u_\psi(x))dm_x(y)d\nu(x) = \frac12\int_{\Omega_m}\int_{\Omega_m}\vert u_\psi(y)-u_\psi(x)\vert dm_x(y)d\nu(x).$$
		Hence,
		$$g(x,y)\in \mbox{sign}(u_\psi(y)-u_\psi(x))\quad (\nu\otimes dm_x)-\mbox{a.e }(x,y)\in \Omega_m\times \Omega_m,$$
		which jointly with \eqref{1-lapla22} and \eqref{1-lapla2456} imply
		that $u$ is a  solution to problem \eqref{1700.intro}. In particular, a single function $g$ determines the structure of all solutions to \eqref{1700.intro}.
	\end{proof}

	\begin{theorem}   Assume that  $(m,\nu)$ satisfies a $p$-Poincar\'{e} inequality in $\Omega$ for all $p > 1$. Let $\psi \in L^\infty(X \backslash  \Omega)$  and $u \in BV_m(X)$ such that $u = \psi$ $\nu$-a.e. on $X \backslash  \Omega$. Then, the following are equivalent:
		\item[ (i)] $u_{\vert \Omega}$ is a minimizer of $\mathcal{J}_{\psi}$.
		\item[ (ii)] $u_{\vert \Omega}$ is a solution of to problem \eqref{1700.intro}.
		\item[ (ii)] $u$ is a function of $m$-least gradient in $\Omega$.
        \item[(iv)] $u_{\vert \Omega_m}$ is a solution of the $m$-least gradient problem \eqref{least1}.
	\end{theorem}

	\begin{proof} By Proposition \ref{equal1} and Theorem~\ref{teo.varia}, we already know that the conditions (i), (ii) and (iv) are
		equivalent.
		
		\noindent (i) implies (iii): Let $g \in BV_m(X)$ such that $g \equiv 0$ $\nu$-a.e. on $X \backslash \Omega$. Them $$TV_m(u) = \mathcal{J}_{\psi} (u_{\vert \Omega}) \leq \mathcal{J}_{\psi} (u_{\vert \Omega}+ g_{\vert \Omega} ) = TV_m(u+g). $$
		Therefore, $u$ is a function of $m$-least gradient in $\Omega$.
		
		\noindent (iii) implies (i): Fixed $v\in BV_m(\Omega)$, we have to see that $\mathcal{J}_{\psi} (u_{\vert \Omega}) \leq \mathcal{J}_{\psi} (v)$. Now, if $g:= v_\psi - u \in BV_m(\Omega_m)$ and $g \equiv 0$ $\nu$-a.e. on $X \backslash \Omega$, we have  $$\mathcal{J}_{\psi} (u_{\vert \Omega}) = TV_m(u) \leq TV_m(u+g) = TV_m(v_\psi) = \mathcal{J}_{\psi} (v).$$
	\end{proof}

%
%
	
	Finally, we note that Theorem \ref{thm:connectionwithlocal} implies that in the case of the random walk $m^{\nu,\varepsilon}$ we have a uniform estimate on the nonlocal gradient on some subsequence (still denoted by $u_\varepsilon$).
	
	\begin{proposition}Let $(X,d,\nu)$ be a metric measure space with doubling measure $\nu$ and supporting a 1-Poincar\'e inequality.
		Let $u_\varepsilon \in BV_{m^{\nu,\varepsilon}}(X)$ be a sequence of solutions to the nonlocal least gradient problem for boundary data $\psi \in BV(X,d,\nu) \cap L^\infty(X, \nu)$. Let $\varepsilon \rightarrow 0$. Then on a subsequence (still denoted by $u_\varepsilon$) we have
		$$ \int_X \dashint_{B(x,\varepsilon)} |(u_\varepsilon)_{\psi}(y) - (u_\varepsilon)_{\psi}(x)| \, d\nu(y) \, d\nu(x) \leq M\varepsilon.$$
	\end{proposition}
	
	\begin{proof}
		Notice that the left hand side of the desired inequality is $2 TV_{m^{\nu,\varepsilon}}(u_\varepsilon)$. As $u_\varepsilon$ is a minimizer of the total variation $TV_{m^{\nu,\varepsilon}}$ with boundary data $\psi$, we have
		$$ 2TV_{m^{\nu,\varepsilon}}(u_\varepsilon) \leq 2TV_{m^{\nu,\varepsilon}}(\psi).$$
		As $\psi \in BV(X)$, by Theorem \ref{thm:connectionwithlocal} on a subsequence (still denoted by $\varepsilon$) for $\varepsilon$ sufficiently close to 0 we have
		$$ 2TV_{m^{\nu,\varepsilon}}(\psi) \leq 4 \varepsilon \cdot \liminf_{\varepsilon \rightarrow 0} \frac{1}{\varepsilon} TV_{m^{\nu,\varepsilon}}(\psi) = M \varepsilon,$$
		where $M = 4 \liminf_{\varepsilon \rightarrow 0} \frac{1}{\varepsilon} TV_{m^{\nu,\varepsilon}}(\psi).$
	\end{proof}
	
\begin{remark}{\rm Let $[\mathbb{R}^N,d,m^J]$ be the metric random walk space of Example \ref{JJ}~(1) and  assume also $J(x)\ge J(y)$ if $|x|\le|y|$. Let  $\Omega$ be a smooth bounded domain in $\mathbb{R}^N$ and $\tilde\psi\in L^{\infty}(\partial \Omega)$. In \cite{MazPeRoTo} it is proved that  if we
 take a function $ \psi \in W^{1,1}(\partial_{m^{J_\epsilon}}\Omega)\cap L^\infty(\partial_{m^{J_\epsilon}}\Omega)$ such that $ \psi|_{\partial \Omega} = \tilde\psi$ and  $u_\epsilon$ is  a  solution of the  problem
	\begin{equation}
	\left\{\begin{array} {ll} \displaystyle  -\Delta^{m^{J_\epsilon}}_1 u(x)=0,& x\in\Omega,
	\\[10pt] u(x)=\psi(x),& x\in \partial_{m^{J_\epsilon}}\Omega,
	\end{array}\right.
	\end{equation}
for $J_{\varepsilon} (x) := \frac{1}{\varepsilon^{N}}
J\left(\frac{x}{\varepsilon}\right)$. Then, up to a subsequence,
$$u_\epsilon\to u\quad\hbox{in }L^1(\Omega),$$ where $u$ is a solution to
\eqref{ELe} with $h=\tilde\psi$.

In particular, if we take $$J(x):= \frac{1}{\mathcal{L}^N(B(0,1))}\1_{B(0,1)}(x),$$ then
$$J_\epsilon(x) = \frac{1}{\mathcal{L}^N(B(0,\epsilon))}\1_{B(0,\epsilon)}(x).$$
Hence,
$$m^{{\mathcal{L}^N,\epsilon}}_x = m^{J_{\epsilon}}_x,$$
and, consequently, if  $u_\epsilon$ is  a  solution of the  problem
	\begin{equation}
	\left\{\begin{array} {ll} \displaystyle  -\Delta^{m^{{\mathcal{L}^N,\epsilon}}}_1 u(x)=0,& x\in\Omega,
	\\[10pt] u(x)=\psi(x),& x\in \partial_{m^{{\mathcal{L}^N,\epsilon}}}\Omega,
	\end{array}\right.
	\end{equation} up to a subsequence,
$$u_\epsilon\to u\quad\hbox{in }L^1(\Omega),$$ where $u$ is a solution to \eqref{ELe} with $h=\tilde\psi$.

Therefore, it is natural to pose the following problem: Let $(X,d, \mu)$ be a metric measure space and let $m^{\mu,\epsilon}$ be the  $\epsilon$-step random walk associated to $\mu$, that is,
 $$m^{\mu,\epsilon}_x:= \frac{\mu \res B(x, \epsilon)}{\mu(B(x, \epsilon))}.$$

Given $f \in L^\infty(\partial \Omega, | D\1_{\overline{\Omega}} |_\mu)$ we consider the  functional $\mathcal{T}_f : L^2(\Omega, \mu) \rightarrow ]-\infty, + \infty]$ defined by
\begin{equation}
\mathcal{T}_f (u):= \left\{ \begin{array}{ll} \vert Du \vert_\mu (\Omega) + \displaystyle\int_{\partial \Omega} \vert u - f \vert \, d | D\1_{\overline{\Omega}} |_\mu\quad &\hbox{if} \ u \in BV(\Omega, d, \mu) \cap L^2(\Omega, \mu), \\ \\ + \infty \quad &\hbox{if} \ u \in  L^2(\Omega, \mu) \setminus BV(\Omega, d, \mu).\end{array}\right.
\end{equation}
We define the multivalued) operator  $-\Delta_{1,\mu}:= \partial \mathcal{T}_f$.

Let $f\in L^\infty(\partial \Omega, | D\1_{\overline{\Omega}} |_\mu)$ such that there exists $ \psi \in W^{1,1}(\partial_{m^{{\mu,\epsilon}}}\Omega)\cap L^\infty(\partial_{m^{{\mu,\epsilon}}}\Omega)$, with  $ \psi|_{\partial \Omega} = f$.
Are there metric measure spaces for which  if  $u_\epsilon$ is  a  solution of the  problem
\begin{equation}
	\left\{\begin{array} {ll} \displaystyle  -\Delta^{m^{{\mu,\epsilon}}}_1 u(x)=0,& x\in\Omega,
	\\[10pt] u(x)=\psi(x),& x\in \partial_{m^{{\mu,\epsilon}}}\Omega,
	\end{array}\right.
	\end{equation} up to a subsequence,
$$u_\epsilon\to u\quad\hbox{in }L^1(\Omega, \mu),$$ where $u$ is a solution to the problem
 \begin{equation}
	\left\{\begin{array} {ll} \displaystyle  -\Delta_{1\mu} u(x)=0,& x\in\Omega,
	\\[10pt] u(x)=f(x),& x\in \partial\Omega.
	\end{array}\right.
\end{equation}
In a forthcoming paper we will study this problem.}
\end{remark}

	\subsection{Nonlocal median value property}
	
	Let us introduce some notation for this subsection. Given a $\nu$-measurable function $u: \Omega \rightarrow \mathbb{R}$, we decompose the space $X$ as
	$$E_+^x = \{ y \in X: u_\psi(y) > u_\psi(x) \}, \qquad E_-^x = \{ y \in X: u_\psi(y) < u_\psi(x) \},$$
	$$E_0^x = \{ y \in X: u_\psi(y) = u_\psi(x) \}.$$
	As $m_x$ is a probability measure, for any $x \in \Omega$ we have
	$$ m_x(E_+^x) + m_x(E_-^x) + m_x(E_0^x) = 1,$$
	so the following two conditions are equivalent:
	\begin{equation}\label{eq:mediancondition1}
	- m_x(E_0^x) \leq m_x(E_-^x) - m_x(E_+^x) \leq m_x(E_0^x)
	\end{equation}
	and
	\begin{equation}\label{eq:mediancondition2}
	m_x(E_+^x \cup E_0^x) \geq \frac12 \quad \mbox{and} \quad m_x(E_-^x \cup E_0^x) \geq \frac12.
	\end{equation}
	\begin{definition}{\rm
			We say that a $\nu$-measurable function $u: \Omega \rightarrow \mathbb{R}$ satisfies the {\it $m$-median value property} if either of the conditions \eqref{eq:mediancondition1} or \eqref{eq:mediancondition2} is satisfied for all $x \in \Omega$.}
	\end{definition}
	
	In the next Theorem, we prove that solutions to \eqref{1700.intro} satisfy the nonlocal median value property. In the other direction, it is not necessarily the case; examples to the contrary exist even in Euclidean spaces, see for instance \cite[Example 3.2]{MazPeRoTo}. Nonetheless, a partial converse still holds.
	
	\begin{theorem}
		Suppose that $u \in BV_m(\Omega)$ is a solution to \eqref{1700.intro}. Then it satisfies the $m$-median value property. Moreover, if $u \in BV_m(\Omega)$ satisfies the $m$-median value property, then it satisfies Definition \ref{Defi.1.var} except for antisymmetry of the function $g$.
	\end{theorem}
	
	\begin{proof}
		Suppose that $u \in BV_m(\Omega)$ is a solution to \eqref{1700.intro}. Take the antisymmetric function $g \in L^\infty(X \times X)$ associated to $u$ given by Definition \ref{Defi.1.var}; in particular, it satisfies
		$$ \int_X g(x,y) dm_x(y) = 0.$$
		We notice that for $y \in E_+^x$ we have $g(x,y) = 1$, for $y \in E_-^x$ we have $g(x,y) = -1$, so
		$$ 0 = \int_X g(x,y) dm_x(y) = \int_{E_+^x} g(x,y) dm_x(y) + \int_{E_-^x} g(x,y) dm_x(y) + $$
		$$ + \int_{E_0^x} g(x,y) dm_x(y) = m_x(E_+^x) - m_x(E_-^x) + \int_{E_0^x} g(x,y) dm_x(y).$$
		As $\| g \|_\infty \leq 1$, we reorder this equality and estimate
		$$ m_x(E_-^x) = m_x(E_+^x) + \int_{E_0^x} g(x,y) dm_x(y) \leq m_x(E_+^x) + m_x(E_0^x) $$
		and
		$$ m_x(E_+^x) = m_x(E_-^x) - \int_{E_0^x} g(x,y) dm_x(y) \leq m_x(E_-^x) + m_x(E_0^x).$$
		We put these two estimates together and obtain \eqref{eq:mediancondition1}, hence $u$ satisfies the nonlocal mean value property. \\
		\\
		In the other direction, it is sufficient to check that $g(x,y)$ defined as
		$$ g(x,y) = \threepartdef{1}{u_\psi(y) > u_\psi(x)}{0}{u_\psi(y) = u_\psi(x)}{-1}{u_\psi(y) < u_\psi(x)}$$
		if $m_x(E_0^x) = 0$ and as
		$$ g(x,y) = \threepartdef{1}{u_\psi(y) > u_\psi(x)}{\frac{m_x(E_-^x) - m_x(E_+^x)}{m_x(E_0^x)}}{u_\psi(y) = u_\psi(x)}{-1}{u_\psi(y) < u_\psi(x)}$$
		if $m_x(E_0^x) > 0$ satisfies Definition \ref{Defi.1.var} except for $g$ being antisymmetric.
	\end{proof}

\subsection{Nonlocal Poincar\'e inequality}
	
In most of the results in Section \ref{sec:lgpnonlocal} we have assumed that $(m,\nu)$ satisfies a $p$-Poincar\'{e} inequality. Let us see now examples of $(m,\nu)$ satisfying  a $p$-Poincar\'{e} inequality.
	
	
We assume the measure $\nu$ to be ergodic, otherwise the Poincar\'e inequality cannot hold - lack of ergodicity of $\nu$ means that in a certain measure-theoretic sense the space is not connected. We will consider examples of metric random walk spaces as introduced in Example \ref{JJ}. Firstly, in (1), the random walk $(m^J, \mathcal{L}^N)$, where $J$ continuous radially symmetric fuction with compact support, satisfies a $p$-Poincar\'{e} inequality (see  \cite{AMRT1} or \cite{ElLibro}). Secondly, in (3), the random walk defined on a finite graph satisfies a $p$-Poincar\'e inequality, as we will see in Corollary \ref{cor:graph}. Thirdly, in (4), the space $(X,d,m^{\nu,\varepsilon})$ under some sensible topological assumptions on the space $X$, satisfies a $p$-Poincar\'e inequality, as we will see in Proposition \ref{prop:poincaredoubling}. Finally, we will see examples of infinite graphs where the Poincar\'e inequality does not hold.
	
\begin{proposition}\label{prop:poincaredoubling}
Suppose that $(X,d, \nu)$ is a length space and that $\Omega$ has finite diameter and the measure $\nu$ is doubling. Then $(m^{\nu, \varepsilon}, \nu)$ satisfies a $p$-Poincar\'e inequality in $\Omega$ for any $p \geq 1$.
\end{proposition}
	
\begin{proof} Let $m_x = m_x^{\nu, \varepsilon}$, $\psi \in L^\infty(X, \nu)$  and we consider:
$$B_0:= \partial_m \Omega, \quad B_1:= \left\{ x \in \Omega \ : d(x, B_0) \leq \frac{\varepsilon}{2} \right\}, \quad B_2:= \left\{ x \in \Omega \setminus B_1 \ : d(x, B_1) \leq \frac{\varepsilon}{2} \right\},$$
$$B_j:= \left\{ x \in \Omega \setminus \bigcup_{k=1}^{j-1} B_k \ : d(x, B_{j-1}) \leq \frac{\varepsilon}{2} \right\}, \quad j=1,2, \ldots .$$
As $\Omega$ has finite diameter, we have
\begin{equation*}
\label{theab01}\exists\,  l \in \mathbb{N} \,:\ \nu\left(\Omega \setminus \bigcup_{j=1}^{l}B_j  \right) = 0.
\end{equation*}
We have
$$\int_{\Omega}	\int_{\Omega_m}  |u_{\psi}(y)-u(x)|^q \, dm_x(y) \, d\nu(x) \geq \int_{B_j}	\int_{B_{j-1}}  |u_{\psi}(y)-u(x)|^q \, dm_x(y) \, d\nu(x), \ \ j=1, \ldots, l.$$
Now, by the reversibility of $\nu$, we have
$$\int_{B_j}
		\int_{B_{j-1}}  |u_{\psi}(y)-u(x)|^q \, dm_x(y) \,
		d\nu(x) \geq $$
		$$\geq \frac{1}{2^q}\int_{B_j}
		\int_{B_{j-1}}  |u(x)|^q \, dm_x(y) \,
		d\nu(x) - \int_{B_j}
		\int_{B_{j-1}}  |u_{\psi}(y)|^q \, dm_x(y) \,
		d\nu(x) = $$ $$= \frac{1}{2^q}\int_{B_j} \left(
		\int_{B_{j-1}}  dm_x(y) \right)\,  |u(x)|^q \,
		d\nu(x) - \int_{B_{j-1}}
		\bigg(\int_{B_{j}} dm_y(x) \bigg) |u_{\psi}(y)|^q \,  \,
		d\nu(y) \geq$$
		$$\geq \frac{1}{2^q}\int_{B_j}
		m_x(B_{j-1})\,  |u(x)|^q \,
		d\nu(x) -
		\int_{B_{j-1}}  |u_{\psi}(y)|^q \,
		d\nu(y).$$
		Now, let us see that $m_x(B_{j-1})$ is bounded from below for $x \in B_j$. To this end, fix any $x \in B_j$ and take the ball $B(x, \varepsilon)$. As $X$ is a length space, there exists a point $y \in B(x, \varepsilon)$ such that $B(y, \frac{\varepsilon}{5}) \subset B(x, \varepsilon) \cap B_{j-1}$. Then, as $\nu$ is doubling and $y \in B(x, \varepsilon)$,  by Proposition \ref{prop:doubling} we have
		$$ m_x(B_{j-1}) \geq m_x \left(B\left(y, \frac{\varepsilon}{5}\right)\right) = \frac{\nu(B(y,\frac{\varepsilon}{5}))}{\nu(B(x, \varepsilon))} \geq C.$$
		We note that this constant does not depend on $\varepsilon$. Then
		$$\int_{\Omega}
		\int_{\Omega_m}  |u_{\psi}(y)-u(x)|^q \, dm_x(y) \,
		d\nu(x) \geq \frac{C}{2^q}\int_{B_j}
		|u(x)|^q \,
		d\nu(x) -
		\int_{B_{j-1}}  |u_{\psi}(y)|^q \,
		d\nu(y).$$
		We rewrite the above inequality as
		$$ \frac{C}{2^q}\int_{B_j} |u(x)|^q \, d\nu(x) \leq \int_{\Omega} \int_{\Omega_m}  |u_{\psi}(y)-u(x)|^q \, dm_x(y) \, 	d\nu(x) + \int_{B_{j-1}}  |u_{\psi}(y)|^q \,
		d\nu(y).$$
		Therefore, since $u_\psi(y) = \psi(y)$ if $y \in B_0$ and $\Omega = \bigcup_{j=1}^l B_j$, we iterate  the last inequality and get that there exists $\lambda >0$ such that
		$$	\lambda\int_\Omega \left|u(x) \right|^q \, d\nu(x)\le \int_{\Omega} \int_{\Omega_m}  |u_{\psi}(y)-u(x)|^q \, dm_x(y) \,	d\nu(x)+\int_{\partial_m\Omega}|\psi (y)|^q \, d\nu(y).$$
		
	\end{proof}
	
\begin{remark}{\rm
In fact, instead of finite diameter of $\Omega$ we may assume finite width, which in the setting of metric measure spaces is exactly the assumption that
$$\Omega \subset \{ x \in X: d(x, \partial_m \Omega) \leq M \}$$
for some $M < \infty$. In this case we also have $\nu(\Omega \backslash \bigcup_{j=1}^l B_j) = 0$ for some $l \in \mathbb{N}$.}
\end{remark}
	
\begin{remark}{\rm
The assumptions on $X$ and $\nu$ include the most typical cases considered in the literature, such as weighted Euclidean spaces, smooth manifolds and Carnot-Carath\'eodory spaces.}
\end{remark}

In fact, we may prove the Poincar\'e inequality in many more settings than only for the random walk $m_x^{\nu,\varepsilon}$. In order for the Poincar\'e inequality to hold, the measure $m_x$ has to be "uniformly distributed around $x$". In the next few results, we present alternative sufficient conditions; the proofs are very similar and the key point is the estimation of $m_x(B_{j-1})$ from below in a uniform way. In particular, we stress that the measure $m_x$ need not be radially distributed nor concentrated around $x$.
	
\begin{proposition} Let $[X,d,m]$ be a metric random walk space with an invariant and reversible measure $\nu$. With assumptions on $X, \Omega$ and $\nu$ as in Proposition \ref{prop:poincaredoubling}, suppose that $m_x$ is a uniformly distributed on an annulus $B(x,\varepsilon) \backslash B(x,\delta)$, i.e.
$$m_x = m_x^{\nu, \varepsilon, \delta} := \frac{\nu \res (B(x,\varepsilon) \backslash \overline{B(x,\delta)})}{\nu(B(x,\varepsilon) \backslash \overline{B(x,\delta)})}.$$
Assume that $\delta = \frac{a}{b} \varepsilon$ with $a,b \in \mathbb{N}$. Then $(m, \nu)$ satisfies a  $p$-Poincar\'e inequality in $\Omega$ for any $p \geq 1$.
\end{proposition}
	
\begin{proof}
Similarly to the proof of Proposition \ref{prop:poincaredoubling}, we set
$$B_1:= \bigg\{ x \in \Omega \ : d(x, \partial_m\Omega) \leq \frac{\varepsilon}{4b} \bigg\}, \quad B_2:= \bigg\{ x \in \Omega \setminus B_1 \ : d(x, B_1) \leq \frac{\varepsilon}{4b} \bigg\},$$
$$B_j:= \left\{ x \in \Omega \setminus \bigcup_{k=1}^{j-1} B_k \ : d(x, B_{j-1}) \leq \frac{\varepsilon}{4b} \right\}, \quad j=2,3, \ldots .$$
As $\Omega$ has finite diameter (or width), we have
\begin{equation*}
\exists\,  l \in \mathbb{N} \,:\ \nu\left(\Omega \setminus \bigcup_{j=1}^{l}B_j \right) = 0.
\end{equation*}
Furthermore, we extend the construction of $B_j$ back into $\partial_m \Omega$ (note that in particular $B_0$ is defined differently to the set in the proof of Proposition \ref{prop:poincaredoubling}). We set
$$ B_{0} = \left\{ x \in \partial_m\Omega: d(x, \Omega) \leq \frac{\varepsilon}{4b} \right\}, \quad B_{-1}:= \left\{ x \in \partial_m\Omega \setminus B_{0} \ : d(x, B_{0}) \leq \frac{\varepsilon}{4b} \right\},$$
$$B_{-j}:= \left\{ x \in \partial_m \Omega \setminus \bigcup_{k=0}^{j-1} B_{-k} \ : d(x, B_{-j+1}) \leq \frac{\varepsilon}{4b} \right\}, \quad j=1,2, \ldots .$$
Now, let $j \geq 1$ and suppose that $x \in B_j$. As for $k \in \mathbb{N}$ we have
$$ \frac{(k-1)\varepsilon}{4b} \leq d(x, B_{j-k}) \leq \frac{k\varepsilon}{4b},$$
we set $j' = j - 4a - 2$ and see that there exists $y \in B(x,\varepsilon)$ such that
$$ B\left(y, \frac{\varepsilon}{16b}\right) \subset (B(x,\varepsilon) \backslash \overline{B(x,\delta)}) \cap B_{j'}.$$
Hence,  by Proposition \ref{prop:doubling} we have that
$$ m_x(B_{j'}) \geq m_x \left(B(y,\frac{\varepsilon}{16b})\right) = \frac{\nu(B(y,\frac{\varepsilon}{16b}))}{\nu(B(x,\varepsilon) \backslash \overline B(x,\delta))} \geq \frac{\nu(B(y,\frac{\varepsilon}{16b}))}{\nu(B(x,\varepsilon))} \geq C = C(b).$$
Hence, we may make a similar calculation as in the proof of Proposition \ref{prop:poincaredoubling} to obtain that for $j=1, \ldots, l$ we have
$$\int_{\Omega} \int_{\Omega_m}  |u_{\psi}(y)-u(x)|^q \, dm_x(y) \,
d\nu(x) \geq \int_{B_j}	\int_{B_{j'}}  |u_{\psi}(y)-u(x)|^q \, dm_x(y) \, d\nu(x) \geq $$
$$\geq \frac{1}{2^q}\int_{B_j}
\int_{B_{j'}}  |u(x)|^q \, dm_x(y) \, d\nu(x) - \int_{B_j} \int_{B_{j'}}  |u_{\psi}(y)|^q \, dm_x(y) \, d\nu(x) = $$
$$= \frac{1}{2^q}\int_{B_j} \left(\int_{B_{j'}}  dm_x(y) \right)\,  |u(x)|^q \,
d\nu(x) - \int_{B_{j'}} \bigg(\int_{B_{j}} dm_y(x) \bigg) |u_{\psi}(y)|^q \,  \, d\nu(y) \geq$$
$$\geq \frac{1}{2^q}\int_{B_j} m_x(B_{j'})\,  |u(x)|^q \, d\nu(x) - \int_{B_{j'}}  |u_{\psi}(y)|^q \, d\nu(y).$$
Since $m_x(B_{j'}) \geq C$, we rewrite this inequality as
$$ \frac{C}{2^q}\int_{B_j} \,  |u(x)|^q \, d\nu(x) \leq \int_{\Omega} \int_{\Omega_m}  |u_{\psi}(y)-u(x)|^q \, dm_x(y) \, d\nu(x) + \int_{B_{j'}}  |u_{\psi}(y)|^q \, d\nu(y).$$
We iterate this result and obtain the Poincar\'e inequality.
\end{proof}
	
\begin{corollary} Let $[X,d,m]$ be a metric random walk space with an invariant and reversible measure $\nu$. With assumptions on $X$ and $\nu$ as in Proposition \ref{prop:poincaredoubling}, suppose that for some $\alpha > 0$ we have $m_x \geq \alpha \, m_x^{\nu, \varepsilon}$ or $m_x \geq \alpha \, m_x^{\nu,\varepsilon,\delta}$. Then the $p$-Poincar\'e inequality holds for any $p \geq 1$.
\end{corollary}
	
\begin{proof}
The only thing that we have to change in the proof of Proposition \ref{prop:poincaredoubling} is the estimate of $m_x(B_{j-1})$ from below. Choose $y \in B(x, \varepsilon)$ as above, so that $B(y, \frac{\varepsilon}{5}) \subset B(x,\varepsilon) \cap B_{j-1}$. Then, we have
$$ m_x(B_{j-1}) \geq m_x\left(B(y, \frac{\varepsilon}{5})\right) \geq \alpha \, m_x^{\nu,\varepsilon}\left(B(y,\frac{\varepsilon}{5})\right) = \alpha \frac{\nu(B(y,\frac{\varepsilon}{5}))}{\nu(B(x, \varepsilon))} \geq \alpha C.$$
We proceed similarly for $m_x \geq \alpha \, m_x^{\nu, \varepsilon,\delta}$.
\end{proof}
In particular, the Corollary above covers the case of random walks $m_x^J$ given by radial weights $J$ defined in Example \ref{exammplee} and the case of measures $m_x^{\nu,\varepsilon,\delta}$ defined on annuli such that $\frac{\varepsilon}{\delta} \notin \mathbb{Q}$.
	
The next Corollary concerns the case of locally finite graphs.
	
\begin{corollary}\label{cor:graph}
Suppose that $G=(V(G),E(G))$ is a locally finite graph and let $(m^G_x, \nu_G)$ be defined as in Example \ref{JJ} (3). Let $\Omega$ be a finite subgraph of $V(G)$. Then $(m^G, \nu_G)$ satisfies a $p$-Poincar\'e inequality in $\Omega$   for any $p \geq 1$.
\end{corollary}
	
\begin{proof}
We set $B_j = \{ x \in \Omega: \text{dist}(x, \partial_m \Omega) = j \}$, where $\text{dist}$ denotes the graph distance. Then, as $\Omega$ is finite, a finite number of $B_j$ covers the whole of $\Omega$ and $\alpha_j = \inf_{x \in B_j} m_x(B_{j-1})$ is strictly positive: as any vertex from $B_j$ has an edge leading to a point in $B_{j-1}$; hence $m_x(B_{j-1})$ is strictly positive for any $x \in B_j$. As the minimisation is over a finite set, $\alpha_j > 0$. Now, we proceed exactly as in the proof of Proposition \ref{prop:poincaredoubling}.
\end{proof}
	
Unfortunately, ergodicity is too weak an assumption to obtain the Poincar\'e inequality. Our main issue with the proof is whether we can define the sets $B_j$ in such a way that $\Omega$ is a finite union of $B_j$'s and we have a uniform bound on $m_x(B_{j-1})$ for $x \in B_j$. We present two examples to highlight this: in the first one, we define a simple metric random walk of type (2) in Example \ref{JJ}. In a second, more involved example, we show that for the case of metric random walk on graphs (as in type (3) in Example \ref{JJ}), if we drop finiteness of the subgraph $\Omega$, the $p$-Poincar\'e inequality may no longer hold, even if we have $\Omega = \{ x \in \Omega: \text{dist}(x, \partial_m \Omega) = 1 \}$.

\begin{example}\label{example:markovchain}{\rm
Let $X = \mathbb{N} \cup \{ 0 \}$. Set $\nu(\{n\}) = 2^{-n}$; in particular, $\nu(X)$ is finite. We define the following random walk $m_x$ on $X$: \\
$\bullet$ $m_0(\{n\}) = 4^{-n}$ for $n \geq 1$; $m_0(\{ 0 \}) = \frac{2}{3}$. \\
$\bullet$ for $n \geq 1$ we have $m_n(\{ 0 \}) = 2^{-n}$; $m_n(\{n \}) = 1 - 2^{-n}$; $m_n(\{k\}) = 0$ for $k \neq n$ with $k \geq 1$.
	
We check that $\nu$ is invariant and reversible with respect to the random walk $m_x$: as our measures are atomic, it is sufficient to check the conditions for singletons. Firstly, we check that
$$ \int_X m_x(\{0\}) d\nu(x) = \frac{2}{3} \cdot 1 + \sum_{n=1}^\infty 2^{-n} \cdot 2^{-n} = 1 = \nu(\{0\}),$$
$$ \int_X m_x(\{n\}) d\nu(x) = (1 - 2^{-n}) 2^{-n} + 4^{-n} \cdot 1 = 2^{-n} = \nu(\{n\}),$$
hence $\nu$ is invariant with respect to $m_x$. Moreover, $\nu$ is reversible:
$$ \int_{\{n\}} \int_{\{n\}} dm_x(y) d\nu(x) = \int_{\{n\}} \int_{\{n\}} dm_y(x) d\nu(y);$$
$$ \int_{\{k\}} \int_{\{n\}} dm_x(y) d\nu(x) = 0 \text{ if } k, n \geq 1, k \neq n;$$
$$ \int_{\{n\}} \int_{\{0\}} dm_x(y) d\nu(x) = 2^{-n} \nu(n) = 4^{-n};$$
$$ \int_{\{n\}} \int_{\{0\}} dm_y(x) d\nu(y) = 4^{-n} \nu(0) = 4^{-n}.$$
Now, we can construct the example. Let $\Omega = \mathbb{N}$. Then $\partial_m \Omega = \{0 \}$, the measure $\nu$ is ergodic and we can go from any two given subsets of positive measure in $X$ (even between any two points) by making two jumps with positive probability; we may say that the $m$-diameter $\text{diam}_m(X)$ is finite and equals $2$. Let us take $\psi \equiv 0$. Then
$$ \int_\Omega |u(x)|^q d\nu(x) = \sum_{n = 1}^\infty 2^{-n} |u(n)|^q.$$
On the other hand, we have
$$ \int_\Omega \int_{\Omega_m} |u_\psi(y) - u(x)|^q dm_x(y) d\nu(x) = \sum_{n = 1}^\infty |u(0) - u(n)|^q \cdot 2^{-n} \cdot 2^{-n} = \sum_{n = 1}^\infty 4^{-n} |u(n)|^q.$$
Hence, a Poincar\'e inequality cannot hold, as we can see for instance from the sequences of the form
$$ (1, 2^{\frac{1}{q}}, 4^{\frac{1}{q}}, ..., 2^{\frac{k}{q}}, 0, 0, 0, ...).$$
The left hand side of the Poincar\'e inequality equals $\lambda k$; the right hand side equals $\sum_{n = 1}^k 4^{-n} \cdot 2^n \leq 1$. As we let $k \rightarrow \infty$, we see that the inequality fails.}
\end{example}

\begin{example}{\rm
Let $X = \{ 2,3,4,...\} \times \{ -1, 1 \}$. For simplicity, we will denote $k := (k,-1)$ and $\overline{k} = (k,1)$. We see it as a graph with vertices in $X$ and with the following edges: \\
$\bullet$ horizontally: from $3n$ to $3n + 1$ with weight $2^{-n}$; the same for $\overline{3n}$ and $\overline{3n+1}$; \\
$\bullet$ horizontally: from $3n+1$ to $3n+2$ with weight $4^{-n}$; the same for $\overline{3n+1}$ and $\overline{3n+2}$; \\
$\bullet$ horizontally: from $3n+2$ to $3n+3$ with weight $2^{-n}$; the same for $\overline{3n+2}$ and $\overline{3n+3}$; \\
$\bullet$ vertically: from $3n$ to $\overline{3n}$ with weight $8^{-n}$; \\
$\bullet$ vertically: from $3n+1$ to $\overline{3n+1}$ with weight $8^{-n}$; \\
$\bullet$ vertically: from $3n+2$ to $\overline{3n+2}$ with weight $8^{-n}$.
	
We set $\nu$ and $m_x$ as for graphs. Then $\nu(k)$ is the degree of the vertex $k$ and we have
$$ \nu(3n) = \nu(\overline{3n}) = d_{3n} = 2^{-n+1} + 2^{-n} + 8^{-n};$$
$$ \nu(3n+1) = \nu(\overline{3n+1}) = d_{3n+1} = 2^{-n} + 4^{-n} + 8^{-n};$$
$$ \nu(3n+2) = \nu(\overline{3n+2}) = d_{3n+2} = 2^{-n} + 4^{-n} + 8^{-n}.$$
	
We set $\Omega = \{ 2,3,4,...\} \times \{ -1 \}$; then $\partial_m \Omega = \mathbb{N} \times \{ 1 \}$. We set $\psi \equiv 0$. Now, suppose that a $q-$Poincar\'e inequality holds. Let us fix $k > 1$ and plug in $u_k \in L^1(X, \nu)$ of the form
$$ u_k(n) = 0 \quad \text{for } n \leq 3k+1, \qquad u_k(n) = 1 \quad \text{for } n \geq 3k+2.$$
The Poincar\'e inequality on graphs is
$$ \lambda\sum_{x \in \Omega} d_x |u(x)|^q \leq \sum_{x \in \Omega} \sum_{y \in \Omega_m} w_{xy} |u_\psi(y) - u(x)|^q + \sum_{y \in \partial_m \Omega} d_y |\psi(y)|^q,$$
hence for such $u_k$ the Poincar\'e inequality reads
$$ LHS = \lambda (2^{-k} + 4^{-k} + 8^{-k}) + \lambda \sum_{n = k+1}^\infty 2 (2^{-n} + 4^{-n} + 8^{-n}) + (2^{-n+1} + 2^{-n} + 8^{-n})$$
and
$$ RHS = 4^{-k} + 4^{-k} + 8^{-k} + \sum_{n=k+1}^\infty 8^{-n}. $$
Hence
$$ \lambda 2^{-k} \leq LHS \leq RHS \leq 3 \cdot 4^{-k}.$$
As we pass with $k \rightarrow \infty$, we see that the Poincar\'e inequality cannot hold for any $\lambda > 0$.}
\end{example}

Of course, failure of the Poincar\'e inequality is due to the fact that $m_x(\partial_m \Omega)$ is unbounded from below. Moreover, these examples also show that when proving the inequality, ergodicity of the random walk $m$ is not sufficient and we cannot work with sets of the form $B_1 = \{ x \in \Omega: m_x(\partial_m\Omega) \geq c \}$ or $B_1 = \{ x \in \Omega: m_x(\partial_m \Omega) > 0 \}$ in place of $B_1 = \{ x \in \Omega: d(x,\partial_m \Omega) \leq \frac{\varepsilon}{2}\}$. This suggests that in order for the Poincar\'e inequality to hold the random walk $m$ has to bear some relation to the topology of the space.

Finally, we take on the question of necessity of Poincar\'e inequality for existence of solutions in the sense of Definition \ref{Defi.1.var}.  In the following Example, we show that if the set $\Omega$ does not support a nonlocal Poincar\'e inequality, then there may be no solutions in the sense of Definition \ref{Defi.1.var}.

\begin{example}{\rm
Let us take the space $X$ introduced in the previous Example and take $\Omega$ as above. Take $u \in BV_m(\Omega)$ and suppose that it satisfies Definition \ref{Defi.1.var} for some $\psi \in L^1(\partial_m \Omega)$. Since the measure $\nu$ is discrete, all properties in Definition \ref{Defi.1.var} are checked pointwise. Let us look at the property \eqref{1-lapla.var} at all points in $\Omega$. Taking $x = 3k+1$, we have
\begin{equation}
0 = \nu(3k+1) \int_{\Omega_m} g(3k+1,y) \, dm_{3k+1}(y) = 2^{-k} g(3k+1,3k) +
\end{equation}
\begin{equation}
+ 4^{-k} g(3k+1,3k+2) + 8^{-k} g(3k+1,\overline{3k+1}).
\end{equation}
We rewrite this as
\begin{equation}\label{eq:finalexample1}
g(3k+1,3k) = - 2^k 4^{-k} g(3k+1,3k+2) - 2^k 8^{-k} g(3k+1,\overline{3k+1}).
\end{equation}
Since $\| g \|_\infty \leq 1$, for $k \geq 1$ we have that $|g(3k+1,3k)| < 1$. Hence, property \eqref{1-lapla2.var} implies that $u(3k+1) = u(3k)$. Similarly, we take $x = 3k-1$ (for $k \geq 2$) and get
\begin{equation}
0 = \nu(3k-1) \int_{\Omega_m} g(3k-1,y) \, dm_{3k-1}(y) = 2^{-k+1} g(3k-1,3k) +
\end{equation}
\begin{equation}
+ 4^{-k+1} g(3k-1,3k-2) + 8^{-k+1} g(3k-1,\overline{3k-1}).
\end{equation}
We rewrite this as
\begin{equation}\label{eq:finalexample2}
g(3k-1,3k) = - 2^{k-1} 4^{-k+1} g(3k-1,3k-2) - 2^{k-1} 8^{-k+1} g(3k-1,\overline{3k-1}).
\end{equation}
Since $\| g \|_\infty \leq 1$, for $k \geq 1$ we have that $|g(3k-1,3k)| < 1$. Hence, property \eqref{1-lapla2.var} implies that $u(3k-1) = u(3k)$. Finally, we take $x=3k$ (for any $k$) and get
\begin{equation}
0 = \nu(3k) \int_{\Omega_m} g(3k,y) \, dm_{3k}(y) = 2^{-k+1} g(3k,3k-1) +
\end{equation}
\begin{equation}
+ 2^{-k} g(3k,3k+1) + 8^{-k+1} g(3k,\overline{3k}).
\end{equation}
We rewrite this as
\begin{equation}\label{eq:finalexample3}
g(3k,3k-1) = - 2^{k-1} 2^{-k} g(3k,3k+1) - 2^{k-1} 8^{-k} g(3k,\overline{3k}).
\end{equation}
In particular, for $k=1$ we have $|g(3,2)| < 1$, so $u(2) = u(3)$. We will combine all these estimates by using the antisymmetry of the function $g$. Specifically, we take equations \eqref{eq:finalexample2} and \eqref{eq:finalexample3} and notice that their left hand side are equal up to the change of the sign. Hence, we have
\begin{equation}
- 2^{k-1} 4^{-k+1} g(3k-1,3k-2) - 2^{k-1} 8^{-k+1} g(3k-1,\overline{3k-1}) = g(3k-1,3k) =
\end{equation}
\begin{equation}
= - g(3k,3k-1) = 2^{k-1} 2^{-k} g(3k,3k+1) + 2^{k-1} 8^{-k} g(3k,\overline{3k}).
\end{equation}
We divide by $2^{k-1}$ and rearrange to obtain that
\begin{equation}
4^{-k+1} g(3k-1,3k-2) = - 2^{-k} g(3k,3k+1) - 8^{-k} g(3k,\overline{3k}) - 8^{-k+1} g(3k-1,\overline{3k-1}).
\end{equation}
Now, divide equation \eqref{eq:finalexample1} by $2^{-k}$. Using the antisymmetry of $g$, we plug it into the equation above. We obtain
\begin{equation}
4^{-k+1} g(3k-1,3k-2) = 4^{-k} g(3k+1,3k+2) + 8^{-k} g(3k+1,\overline{3k+1}) +
\end{equation}
\begin{equation}
- 8^{-k} g(3k,\overline{3k}) - 8^{-k+1} g(3k-1,\overline{3k-1}),
\end{equation}
which we then rearrange to
\begin{equation}
g(3k+1,3k+2) = 4 g(3k-1,3k-2) - 4^k 8^{-k} (g(3k+1,\overline{3k+1}) +
\end{equation}
\begin{equation}
+ g(3k,\overline{3k}) + 8 g(3k-1,\overline{3k-1})).
\end{equation}
Assume that $g(3k-1,3k-2) = \pm 1$. Then, since $\| g \|_\infty \leq 1$, for $k \geq 2$ we have that $|g(3k+1,3k+2)| \geq 4 - 10 \cdot 4^k 8^{-k} \geq \frac32$, contradiction with $\| g \|_\infty \leq 1$. Hence, for $k \geq 2$ we have $u(3k-1) = u(3k-2)$.

We use all the obtained equalities to see that $u$ is constant on $\Omega$; for any $k$, we have $u(3k-1) = u(3k) = u(3k+1)$. For $k \geq 2$, we also have $u(3k-1) = u(3k-2)$; but this covers all points of this form in $\Omega$, since for $k=1$ there is no corresponding point in $\Omega$. Moreover, we can easily compute the value of this constant function. Let us take in \eqref{1-lapla.var} the point $x = 2$. We obtain
\begin{equation}
0 = \nu(2) \int_{\Omega_m} g(2,y) \, dm_{2}(y) = g(2,3) + g(2,\overline{2}).
\end{equation}
Assume that $g(2,\overline{2}) = \pm 1$; then we have $g(2,3) = \mp 1$. Now, we take in \eqref{1-lapla.var} the point $x = 3$, so by \eqref{eq:finalexample3} and antisymmetry of $g$ we obtain
\begin{equation}
|g(2,3)| = |-g(3,2)| = |\frac12 g(3,4) + \frac18 g(3,\overline{3})| \leq \frac58,
\end{equation}
contradiction. Hence, $|g(2,\overline{2})| < 1$, so property \eqref{1-lapla2.var} implies that $u \equiv u(2) = u(\overline{2})$.

Finally, we see that this means that the nonlocal least gradient problem in the sense of Definition \ref{Defi.1.var} has no solution in $\Omega$. Assume that a solution $u$ exists. By using the antisymmetry of $g$ and taking property \eqref{1-lapla.var} at points $x = 3k, 3k+1, 3k+2$ we get respectively
$$g(3k, 3k+1) = 2g(3k-1,3k) - 4^{-k} g(3k, \overline{3k});$$
$$g(3k+1, 3k+2) = 2^k g(3k,3k+1) - 2^{-k} g(3k+1, \overline{3k+1});$$
$$g(3k+2, 3k+3) = 2^{-k} g(3k+1,3k+2) - 4^{-k} g(3k+2, \overline{3k+2}).$$
We iterate these results to obtain
$$g(3k+2, 3k+3) = 2 g(3k-1,3k) - 4^{-k} (g(3k, \overline{3k}) + g(3k+1, \overline{3k+1}) + g(3k+2, \overline{3k+2}))$$
and similar formulas for the other cases. Assume that $g(3k+2,3k+3) \neq 0$ for some $k$; take $N$ large enough so that $|g(3k+2,3k+3)| \geq \frac{1}{2^{N-2}}$. Then, by the iterative formula above, we have
$$|g(3k+3N+2, 3k+3N+3)| \geq 2^N |g(3k+2,3k+3)| - 3 \sum_{n=k+1}^\infty 4^{-k+1} \geq 4 - 1 = 3,$$
contradiction. Hence, for all $k$ we have $g(3k+2,3k+3) = 0$. We deal similarly with other cases and see that $g(n,n+1) = 0$ for all $n$; however, if we take nonconstant boundary data $\psi$, since the solution is constant, this contradicts property \eqref{1-lapla2.var}. Hence, there are no solutions to the nonlocal least gradient problem in the sense of Definition \ref{Defi.1.var}.}
\end{example}

In this Example, the reason that solutions did not exist is not only the failure of the Poincar\'e inequality, but also the fact that the connections between $\Omega$ and $\partial_m \Omega$ have very low weights, so $\Omega$ and $\partial_m \Omega$ are effectively disconnected if one tries to minimize the functional $\mathcal{J}_\psi$ and the solution is forced to be constant. This phenomenon does not appear in Example \ref{example:markovchain}, where we can compute the solutions by hand.  Moreover, we extensively used antisymmetry of $g$; we point out that there exist minima of the functional $\mathcal{J}_\psi$, solutions in the sense of functions of $m$-least gradient, and functions that satisfy the $m$-median value property. The thing that fails is the Euler-Lagrange characterisation of these objects, for which we need a Poincar\'e inequality.

	\bigskip
	\noindent {\bf Acknowledgments.} The first author has been partially supported by the research project no. 2017/27/N/ST1/02418 funded by the National Science Centre, Poland, and by Integrated Development Programme of the University of Warsaw, co-financed by the European Social Fund via Operational Programme Knowledge Education Development 2014-2020, path 3.5. The first author also wishes to thank Universitat de Val\`encia for their hospitality. The second author has been partially supported by the Spanish MCIU and FEDER, project PGC2018-094775-B-100.

\end{document}